\documentclass{birkjour}
\usepackage{hyperref,url}

 \newtheorem{thm}{Theorem}[section]
 \newtheorem{cor}[thm]{Corollary}
 \newtheorem{lem}[thm]{Lemma}
 \numberwithin{equation}{section}

\newcommand{\fa}{\mathfrak{a}}
\newcommand{\fb}{\mathfrak{b}}
\newcommand{\fc}{\mathfrak{c}}
\newcommand{\fe}{\mathfrak{e}}
\newcommand{\fg}{\mathfrak{g}}
\newcommand{\fr}{\mathfrak{r}}
\newcommand{\fs}{\mathfrak{s}}

\newcommand{\cA}{\mathcal{A}}
\newcommand{\cB}{\mathcal{B}}
\newcommand{\cC}{\mathcal{C}}
\newcommand{\cZ}{\mathcal{Z}}
\newcommand{\cD}{\mathcal{D}}
\newcommand{\cE}{\mathcal{E}}
\newcommand{\cF}{\mathcal{F}}

\newcommand{\cI}{\mathcal{I}}
\newcommand{\cJ}{\mathcal{J}}
\newcommand{\cK}{\mathcal{K}}
\newcommand{\cN}{\mathcal{N}}

\newcommand{\mC}{\mathbb{C}}
\newcommand{\mI}{\mathbb{I}}
\newcommand{\mN}{\mathbb{N}}
\newcommand{\mR}{\mathbb{R}}

\newcommand{\fA}{\mathfrak{A}}
\newcommand{\fJ}{\mathfrak{J}}

\newcommand{\fZ}{\mathfrak{Z}}

\newcommand{\alg}{\operatorname{alg}}
\newcommand{\clos}{\operatorname{clos}}
\newcommand{\eps}{\varepsilon}

\begin{document}
%
%
%
%
%
%

\title[Singular integral operators with shifts]
{Sufficient Conditions for Fredholmness\\
of Singular Integral Operators with Shifts\\
and Slowly Oscillating Data}

\author[A. Yu. Karlovich]{Alexei Yu. Karlovich}
\address{
Departamento de Matem\'atica\\
Faculdade de Ci\^encias e Tecnologia\\
Universidade Nova de Lisboa\\
Quinta da Torre\\
2829--516 Caparica\\
Portugal}
\email{oyk@fct.unl.pt}

\author[Yu. I. Karlovich]{Yuri I. Karlovich}
\address{%
Facultad de Ciencias \\
Universidad Aut\'onoma del Estado de Morelos\\
Av. Universidad 1001, Col. Chamilpa,\\
C.P. 62209 Cuernavaca, Morelos, \\
M\'exico}
\email{karlovich@uaem.mx}

\author[A. B. Lebre]{Amarino B. Lebre}
\address{%
Departamento de Matem\'atica \\
Instituto Superior T\'ecnico \\
Universidade T\'ecnica de Lisboa\\
Av. Rovisco Pais, \\
1049--001 Lisboa \\
Portugal}
\email{alebre@math.ist.utl.pt}

\thanks{This work is partially supported by ``Centro de An\'alise Funcional e Aplica\c{c}\~oes"
at Instituto Superior T\'ecnico (Lisboa, Portugal), which is financed by FCT (Portugal).
The second author is also supported by the SEP-CONACYT Project No. 25564 (M\'exico)
and by PROMEP (M\'exico) via ``Proyecto de Redes".}



\date{September 27, 2010}


\begin{abstract}
Suppose $\alpha$ is an orientation preserving diffeomorphism
(shift) of $\mR_+=(0,\infty)$  onto itself with the only fixed
points $0$ and $\infty$. We establish sufficient conditions for the
Fredholmness of the singular integral operator
\[
(aI-bW_\alpha)P_++(cI-dW_\alpha)P_-
\]
acting on $L^p(\mR_+)$ with $1<p<\infty$, where $P_\pm=(I\pm S)/2$,
$S$ is the Cauchy singular integral
operator, and $W_\alpha f=f\circ\alpha$ is the shift operator,
under the assumptions that the coefficients $a,b,c,d$ and the
derivative $\alpha'$ of the shift are bounded and continuous on
$\mR_+$ and may admit discontinuities of slowly oscillating type
at $0$ and $\infty$.
\end{abstract}
\maketitle
\section{Introduction}
Let $\cB(X)$ denote the Banach algebra of all bounded linear
operators acting on a Banach space $X$, let $\cK(X)$ be the closed
two-sided ideal of all compact operators in $\cB(X)$, and let
$\cB^\pi(X):=\cB(X)/\cK(X)$ be the Calkin algebra of the cosets
$A^\pi:=A+\cK(X)$ where $A\in\cB(X)$. An operator $A\in\cB(X)$ is
said to be \textit{Fredholm} if its image is closed and the spaces
$\ker A$ and $\ker A^*$ are finite-dimensional.

Through this paper we will assume that $1<p<\infty$. Let
$C_b(\mR_+)$ denote the $C^*$-algebra of all bounded continuous
functions on $\mR_+:=(0,+\infty)$, and let $\alpha$ be an
orientation-preserving diffeomorphism of $\mR_+$ onto itself,
which has only two fixed points $0$ and $\infty$. The function
$\alpha$ is referred to as an {\it orientation-preserving
non-Carleman shift} on $\mR_+$. If $\log\alpha'\in C_b(\mR_+)$,
then the shift operator $W_\alpha$, given by $W_\alpha
f=f\circ\alpha$, is an isomorphism of the Lebesgue space
$L^p(\mR_+)$ onto itself. As is well known, the Cauchy singular
integral operator $S$ given by
\[
(Sf)(t):=\lim_{\varepsilon\to 0} \frac{1}{\pi i}
\int_{\mR_+\setminus(t-\varepsilon,t+\varepsilon)}
\frac{f(\tau)}{\tau-t}\:d\tau\quad(t\in\mR_+)
\]
is bounded on the Lebesgue space $L^p(\mR_+)$. Then the
operators $P_\pm:=\frac{1}{2}(I\pm S)$ also are in
$\cB(L^p(\mR_+))$.

The Fredholm theory of singular integral operators with
discontinuous coefficients and shifts on Lebesgue spaces has the
long and rich history. We mention the monographs by Gohberg and
Krupnik \cite{GK92}, Mikhlin and Pr\"ossdorf \cite{MP86}, and
B\"ottcher and the second author \cite{BK97} for the Fredholm
theory of singular integral operators with jump discontinuities
(and without shifts); the books by
Litvinchuk \cite{L77},
Roch and Silbermann \cite{RS90},
Kravchenko and Litvinchuk \cite{KL94},
Antonevich \cite{A95},
Karapetiants and Samko \cite{KS01},
and the~references therein for the Fredholm theory of singular integral
operators with shifts. In all these sources coefficients of singular
integral operators and derivatives of shifts are supposed to be
either continuous or piecewise continuous.

Singular integral operators with coefficients admitting
discontinuities of slowly oscillating type were considered in
\cite{R96,R98,BKR98,BKR00}. We also mention
the works \cite{BFK06,BFK07,BFK08}, where $C^*$-algebras of
singular integral operators with piecewise slowly oscillating
coefficients and various classes of shifts with continuous
derivatives were considered in the setting of $L^2$-spaces.
Singular integral ope\-rators with shifts and slowly oscillating
data were studied in \cite{K08,K08-AMADE,K09,KL01} by applying the theory
of pseudodifferential operators with so-called compound (double)
non-regular symbols. This approach requires at least three times
continuous differentiability of slowly oscillating shifts.

This paper is devoted to singular integral operators with slowly
oscillating coefficients and slowly oscillating non-Carleman
shifts preserving the orientation in the $L^p(\mR_+)$-setting.
Our results generalize and complement those of \cite{KK81} (see
also \cite[Chap.~4, Section~2]{KL94}), where a Fredholm
criterion was obtained for a singular integral operator with
continuous coefficients and a shift being an orientation-preserving
diffeomorphism of $[0,1]$ onto itself with the only fixed points
$0$ and $1$. In contrast to previous applications of pseudodifferential
operators, our approach here is based on a simpler theory of
pseudodifferential operators with non-compound non-regular symbols
combined with the Allan-Douglas local principle (see \cite{BS06})
and invertibility results for functional operators. This allows us
to reduce the smoothness of shifts to the existence of slowly
oscillating first derivatives.

To formulate our main results explicitly, we need several
definitions. Following \cite{{S77}}, a function $f\in C_b(\mR_+)$
is called {\it slowly oscillating} (at $0$ and $\infty$) if for
each (equivalently, for some) $\lambda\in (0,1)$,
\[
\lim_{r\to s}\operatorname{osc}(f,[\lambda r,r])=0\quad
(s\in\{0,\infty\}),
\]
where
$
\operatorname{osc}(f,[\lambda r,r]):=
\sup\big\{|f(t)-f(\tau)|:t,\tau\in[\lambda r,r]\big\}
$
is the oscillation of $f$ on the segment
$[\lambda r,r]\subset\mR_+$. Obviously, the set $SO(\mR_+)$
of all slowly oscillating (at $0$ and $\infty$) functions in
$C_b(\mR_+)$ is a unital commutative $C^*$-algebra. This algebra
properly contains $C(\overline{\mR}_+)$, the $C^*$-algebra
of all continuous functions on $\overline{\mR}_+:=[0,+\infty]$.

We say that an orientation-preserving non-Carleman shift $\alpha$
is {\it slowly oscillating} (at $0$ and $\infty$) if
$\log\alpha'\in C_b(\mR_+)$ and $\alpha'\in SO(\mR_+)$. We denote
by $SOS(\mR_+)$ the set of such shifts. As we will show in
Section~\ref{sec:exponential-representation}, the shifts
$\alpha\in SOS(\mR_+)$ are represented in the form
$\alpha(t)=te^{\omega(t)}$ for $t\in\mR_+$ where $\omega\in
SO(\mR_+)$.

Our first concern is the invertibility of binomial functional
operators with slowly oscillating coefficients and a slowly
oscillating shift. The following theorem
was obtained in the $L^p(0,1)$-setting  in
\cite{KKL03}, the present version is derived
from that one in Section~\ref{sec:proof-FO}.
\begin{thm}\label{th:FO}
Suppose $a,b\in SO(\mR_+)$ and $\alpha\in SOS(\mR_+)$. The
functional operator $aI-bW_\alpha$ is invertible on the
Lebesgue space $L^p(\mR_+)$ if and only if
either
\begin{equation}\label{eq:FO-1}
\inf\limits_{t\in\mR_+}|a(t)|>0, \
\liminf\limits_{t\to s}
\left( |a(t)|-|b(t)|\big(\alpha'(t)\big)^{-1/p}\right)>0 \
(s\in\{0,\infty\});
\end{equation}
or
\begin{equation}\label{eq:FO-2}
\inf\limits_{t\in\mR_+}|b(t)|>0, \
\limsup\limits_{t\to s}
\left(|a(t)|-|b(t)|\big(\alpha'(t)\big)^{-1/p}\right)<0 \
(s\in\{0,\infty\}).
\end{equation}
If \eqref{eq:FO-1} holds, then
\begin{equation}\label{eq:FO-3}
(aI-bW_\alpha)^{-1}=\sum_{n=0}^\infty (a^{-1}bW_\alpha)^n a^{-1}I.
\end{equation}
If \eqref{eq:FO-2} holds, then
\begin{equation}\label{eq:FO-4}
(aI-bW_\alpha)^{-1}=-W_\alpha^{-1}\sum_{n=0}^\infty (b^{-1}aW_\alpha^{-1})^n b^{-1}I.
\end{equation}
\end{thm}
By $M(\mathfrak{A})$ denote the maximal ideal space of a unital commutative
Banach algebra $\mathfrak{A}$. Identifying the points
$t\in\overline{\mR}_+$ with the evaluation functionals $t(f)=f(t)$
for $f\in C(\overline{\mR}_+)$, we get
$M(C(\overline{\mR}_+))=\overline{\mR}_+$. Consider the fibers
\[
M_s(SO(\mR_+)):=\big\{\xi\in M(SO(\mR_+)):\xi|_{C(\mR_+)}=s\big\}
\]
of the maximal ideal space $M(SO(\mR_+))$ over the points
$s\in\{0,\infty\}$. By \cite[Proposition~2.1]{K08}, the set
\begin{equation}\label{eq:a1}
\Delta:=M_0(SO(\mR_+))\cup M_\infty(SO(\mR_+))
\end{equation}
coincides with $\operatorname{clos}_{SO^*}\mR_+\setminus\mR_+$
where $\operatorname{clos}_{SO^*}\mR_+$ is the weak-star closure
of $\mR_+$ in the dual space of $SO(\mR_+)$. Then $M(SO(\mR_+))=\Delta\cup\mR_+$.
In what follows we write $a(\xi):=\xi(a)$
for every $a\in SO(\mR_+)$ and every $\xi\in\Delta$.

We now formulate the main result of the paper.
\begin{thm}\label{th:main}
Suppose $a,b,c,d\in SO(\mR_+)$ and $\alpha\in SOS(\mR_+)$. The
singular integral operator
\begin{equation}\label{eq:def-N}
N:= (aI-bW_\alpha)P_++(cI-dW_\alpha)P_-
\end{equation}
with the shift $\alpha$ is Fredholm on the space $L^p(\mR_+)$ if
the following two conditions are fulfilled:
\begin{enumerate}
\item[(i)] the functional operators $A_+:=a I-bW_\alpha$ and
$A_-:=cI-dW_\alpha$ are invertible on the space $L^p(\mR_+)$;

\item[(ii)] for every pair $(\xi,x)\in\Delta\times\mR$,
\begin{align}
n_\xi(x)
&:=
\Big[a(\xi)-b(\xi)e^{i\omega(\xi)(x+i/p)}\Big]
\frac{1+\coth[\pi(x+i/p)]}{2}
\nonumber
\\
&\quad+
\Big[c(\xi)-d(\xi)e^{i\omega(\xi)(x+i/p)}\Big]
\frac{1-\coth[\pi(x+i/p)]}{2}\ne 0,
\label{eq:def-n}
\end{align}
where $\omega(t):=\log[\alpha(t)/t]\in SO(\mR_+)$.
\end{enumerate}
\end{thm}
The paper is organized as follows. In
Section~\ref{sec:SO-properties} we collect properties of slowly
oscillating functions and shifts. In particular, we prove that
each slowly oscillating shift can be represented in the form
$\alpha(t)=te^{\omega(t)}$ where $\omega$ is a real-valued slowly
oscillating function. Section~\ref{sec:FO-binomial} is devoted to
the proof of Theorem~\ref{th:FO}.
In Section~\ref{sec:Mellin-convolution} we collect
properties of Mellin convolution operators with piecewise
continuous symbols. In particular, we
mention the well-known structure of the commutative algebra $\cA$
generated by the operator $S$ and the identity operator. This
algebra contains the operator with fixed singularities
\[
(Rf)(t)=\frac{1}{\pi i}\int_{\mR_+}\frac{f(\tau)}{\tau+t}d\tau
\quad(t\in\mR_+).
\]
Section~\ref{sec:Mellin-PDO} contains necessary facts from the
theory of Mellin pseudodifferential operators with slowly
oscillating symbols.

Section~\ref{sec:localization} is dedicated to the localization with
the aid of the Allan-Douglas local principle. We introduce the
algebra $\cZ$ generated by the compact operators and the operators
$I$, $S$, and $cR$, where $c$ are slowly oscillating functions.
Further we introduce the algebra $\Lambda$ of operators commuting
with the elements of $\cZ$ modulo compact operators. The
Fredholmness of an operator $A\in\Lambda$ is equivalent to the
invertibility of the coset $A^\pi$ in the quotient algebra
$\Lambda^\pi=\Lambda/\cK$. The maximal ideal space of its central
subalgebra $\cZ^\pi=\cZ/\cK$ is homeomorphic to the set
$\{-\infty,+\infty\}\cup(\Delta\times\mR)$. Since $N\in\Lambda$,
by the Allan-Douglas local principle, the Fredholmness of $N$ is
equivalent to the invertibility of the local representatives
$N^\pi+\cJ_{-\infty}^\pi$, $N^\pi+\cJ_{+\infty}^\pi$, and
$N^\pi+\cJ_{\xi,x}^\pi$ in the local algebras
$\Lambda^\pi/\cJ_{-\infty}^\pi$, $\Lambda^\pi/\cJ_{+\infty}^\pi$,
and $\Lambda^\pi/\cJ_{\xi,x}^\pi$ with
$(\xi,x)\in\Delta\times\mR$, respectively, where $\cJ_{\pm\infty}^\pi$ and
$\cJ_{\xi,x}^\pi$ are ideals of $\Lambda^\pi$.

In Section~\ref{sec:functions-in-fC} we prove that certain functions,
playing an important role in the proof of the sufficiency portion of
Theorem~\ref{th:main}, belong to the algebra of symbols of
Mellin pseudodifferential operators commuting modulo compact operators.
Section~\ref{sec:sufficiency} is dedicated to the proof of Theorem~\ref{th:main}.
The invertibility of the functional operators $aI-bW_\alpha$ and $cI-dW_\alpha$ imply
the invertibility of the cosets $N^\pi+\cJ_{+\infty}^\pi$ and
$N^\pi+\cJ_{-\infty}^\pi$ in the local algebras $\Lambda^\pi/\cJ_{+\infty}^\pi$
and $\Lambda^\pi/\cJ_{-\infty}^\pi$, respectively. On the other hand,
using the technique of Mellin pseudodifferential operators
we show that $n_\xi(x)\ne 0$ implies that the coset $N^\pi+\cJ_{\xi,x}^\pi$
is invertible in the local algebra $\Lambda^\pi/\cJ_{\xi,x}^\pi$ for
$(\xi,x)\in\Delta\times\mR$. The proof is based on the important property:
the product $W_\alpha R$ of the shift operator $W_\alpha$ and the operator $R$
with fixed singularities at $0$ and $\infty$ is similar to a Mellin
pseudodifferential operator. To finish the proof of Theorem~\ref{th:main},
it remains to apply the Allan-Douglas local principle.

Finally, we note that conditions (i) and (ii) of Theorem~\ref{th:main}
are also necessary for the Fredholmness of the operator $N$. This statement
will be proved in the forthcoming paper \cite{KKLnecessity}.
\section{Slowly oscillating functions and shifts}\label{sec:SO-properties}
\subsection{Fundamental property of slowly oscillating functions}
\begin{lem}[{\cite[Proposition~2.2]{K08}}]
\label{le:SO-fundamental-property}
Let $\{a_k\}_{k=1}^\infty$ be a countable subset of $SO(\mR_+)$ and
$s\in\{0,\infty\}$. For each $\xi\in M_s(SO(\mR_+))$ there exists a
sequence $\{t_n\}\subset\mR_+$ such that $t_n\to s$ as $n\to\infty$ and
\begin{equation}\label{eq:SO-fundamental-property}
\xi(a_k)=\lim_{n\to\infty}a_k(t_n)\quad\mbox{for all}\quad k\in\mN.
\end{equation}
Conversely, if $\{t_n\}\subset\mR_+$ is a sequence such that $t_n\to s$
as $n\to\infty$, then there exists a functional $\xi\in M_s(SO(\mR_+))$
such that \eqref{eq:SO-fundamental-property} holds.
\end{lem}
\subsection{Exponential representation of slowly oscillating shifts}
\label{sec:exponential-representation}
\begin{lem}\label{le:exp-repr}
An orientation-preserving non-Carleman shift
$\alpha:\mR_+\to\mR_+$ belongs to $SOS(\mR_+)$ if and only if
\begin{equation}\label{eq:exp-repr-1}
\alpha(t)=te^{\omega (t)},\quad t\in \mR_+,
\end{equation}
for some real-valued function $\omega\in SO(\mR_+)\cap C^1(\mR)$ such that
the function $t\mapsto t\omega^\prime(t)$ also belongs to $SO(\mR_+)$ and
\begin{equation}\label{eq:exp-repr-2}
\inf_{t\in\mR_+}\big(1+t\omega'(t)\big)>0.
\end{equation}
\end{lem}
\begin{proof}
\textit{Necessity.} Let $\alpha\in SOS(\mR_+)$. Then
$\log\alpha'\in C_b(\mR_+)$ and hence
\begin{equation}\label{eq:exp-repr-3}
0<m_\alpha:=\inf_{t\in\mR_+}\alpha'(t)\le\sup_{t\in\mR_+}\alpha'(t)=:M_\alpha<\infty,
\end{equation}
$\alpha'\in SO(\mR_+)$, and $\alpha(0)=0$,
$\alpha(\infty)=\infty$. As $\alpha(0)=0$, we have the
representation
\[
\frac{\alpha(t)}{t}=\frac{1}{t}\int_0^t\alpha'(x)dx=\int_0^1\alpha'(tx)dx,
\]
which implies due to \eqref{eq:exp-repr-3} that
\begin{equation}\label{eq:exp-repr-4}
0<m_\alpha\le\inf_{t\in\mR_+}\frac{\alpha(t)}{t}\le\sup_{t\in\mR_+}
\frac{\alpha(t)}{t}\le M_\alpha<\infty
\end{equation}
and
\begin{equation}\label{eq:exp-repr-5}
\frac{\alpha(t)}{t}-\frac{\alpha(\tau)}{\tau}
=
\int_0^1\big(\alpha'(tx)-\alpha'(\tau x)\big)dx.
\end{equation}
Using \eqref{eq:exp-repr-4} and \eqref{eq:exp-repr-5}, we conclude that the function
$D(t):=\alpha(t)/t$ belongs to $C_b(\mR_+)$. Furthermore,
from \eqref{eq:exp-repr-5} it follows that
\begin{align}
\operatorname{osc}(D,[r/2,r])
&\le
\int_0^1\sup_{t,\tau\in[r/2,r]}\big|\alpha'(tx)-\alpha'(\tau x)\big|dx
\nonumber
\\
&=
\int_0^1\operatorname{osc}(\alpha',[rx/2,rx])dx.
\label{eq:exp-repr-6}
\end{align}
Since $\alpha'\in SO(\mR_+)$, we conclude that for every
$\eps>0$ there exist positive numbers $\delta_0<\delta_\infty$
such that $\operatorname{osc}(\alpha',[r/2,r])<\eps$ for all
$r\in(0,\delta_0)\cup(\delta_\infty,\infty)$.
Hence, for $r\in(0,\delta_0)$ and all $x\in(0,1]$,
\begin{equation}\label{eq:exp-repr-7}
\operatorname{osc}(\alpha',[rx/2,rx])<\eps,
\end{equation}
which implies due to \eqref{eq:exp-repr-6} that
\begin{equation}\label{eq:exp-repr-8}
\lim_{r\to 0}\operatorname{osc}(D,[r/2,r])=0.
\end{equation}
On the other hand, for $r>\delta_\infty$ and all
$x\in(\delta_\infty/r,1]$, we also have \eqref{eq:exp-repr-7}.
Therefore,
\[
\begin{split}
&
\int_{\delta_\infty/r}^1
\operatorname{osc}(\alpha',[rx/2,rx])dx
\le
\eps(1-\delta_\infty/r),
\\
&
\int_0^{\delta_\infty/r}
\operatorname{osc}(\alpha',[rx/2,rx])dx
\le
2\|\alpha'\|_{L^\infty(\mR_+)}\delta_\infty/r,
\end{split}
\]
whence
\[
\int_0^1 \operatorname{osc}(\alpha',[rx/2,rx])dx
\le
\eps(1-\delta_\infty/r)+2\|\alpha'\|_{L^\infty(\mR_+)}\delta_\infty/r.
\]
This implies in view of \eqref{eq:exp-repr-6} that
\begin{equation}\label{eq:exp-repr-9}
\lim_{r\to\infty}\operatorname{osc}(D,[r/2,r])=0.
\end{equation}
Thus, by \eqref{eq:exp-repr-8} and \eqref{eq:exp-repr-9}, the function
$D(t)=\alpha(t)/t$ actually belongs to $SO(\mR_+)$.

Since $SO(\mR_+)$ is a $C^*$-algebra, we infer from \eqref{eq:exp-repr-4}
that the function $\omega(t):=\log[\alpha(t)/t]$ also is in
$SO(\mR_+)$. Thus $\alpha(t)=te^{\omega(t)}$ where $\omega'\in
C(\mR_+)$. Finally, since $\alpha',\ e^\omega\in SO(\mR_+)$, it
follows from the equality
\begin{equation}\label{eq:exp-repr-10}
\alpha'(t)=\big(1+t\omega'(t)\big)e^{\omega(t)}
\end{equation}
and \eqref{eq:exp-repr-3}--\eqref{eq:exp-repr-4} that the function $t\mapsto
t\omega'(t)$ also belongs to $SO(\mR_+)$ and \eqref{eq:exp-repr-2} holds.

\textit{Sufficiency.} Let $\alpha$ be an orientation-preserving
diffeomorphism of $\mR_+$ onto itself, with the fixed points $0$
and $\infty$ only; and let $\alpha(t)=te^{\omega (t)}$, where the
functions $\omega$ and $t\mapsto t\omega'(t)$ are in $SO(\mR_+)$
and \eqref{eq:exp-repr-2} holds. Since
\[
0<\inf_{t\in\mR_+}e^{\omega(t)}\le\sup_{t\in\mR_+}e^{\omega(t)}<\infty,
\]
we infer from \eqref{eq:exp-repr-2} and \eqref{eq:exp-repr-10} that
$\log\alpha'\in C_b(\mR_+)$. Furthermore, as the functions $e^\omega$ and
$t\mapsto t\omega'(t)$ are in $SO(\mR_+)$, from \eqref{eq:exp-repr-10} it
follows that $\alpha'$ belongs to $SO(\mR_+)$. Thus, $\alpha\in SOS(\mR_+)$.
\end{proof}
The representation \eqref{eq:exp-repr-1} will be called the
\textit{exponential representation of the slowly oscillating shift
$\alpha$} and the function $\omega$ will be referred to as the
\textit{exponent function of $\alpha$}.
\subsection{Properties of slowly oscillating shifts}
\begin{lem}\label{le:continuous-SOS}
If $c\in SO(\mR_+)$ and $\alpha\in SOS(\mR_+)$, then $c\circ\alpha\in SO(\mR_+)$,
$c-c\circ\alpha\in C_b(\mR_+)$, and
$\lim\limits_{t\to s}\big(c(t)-c(\alpha(t))\big)=0$ for $s\in\{0,\infty\}$.
\end{lem}
\begin{proof}
Obviously, $c-c\circ\alpha\in C_b(\mR_+)$.
Since $\alpha\in SOS(\mR_+)$, we deduce from \eqref{eq:exp-repr-4} that
$m_\alpha t\le\alpha(t)\le M_\alpha t$ for every $t\in\mR_+$,
where the positive numbers $m_\alpha\le M_\alpha$ are defined in
\eqref{eq:exp-repr-3}. Hence, for every $t\in \mR_+$ we obtain
\begin{equation}\label{eq:continuous-SOS-1}
\operatorname{osc}(c\circ\alpha,[t/2,t])\le\operatorname{osc}(c,[(\lambda/2)r,r]),
\;
|c(t)-c(\alpha(t))|\le \operatorname{osc}(c,[\lambda r,r]),
\end{equation}
where $\lambda r=t\min\{m_\alpha,1\}$ and $r=t\max\{M_\alpha,1\}$.
Therefore we conclude that
$\lambda=\min\{m_\alpha,1\}/\max\{M_\alpha,1\}\in(0,1)$ except for the
trivial case $\alpha(t)=t$.
Since $c\in SO(\mR_+)$, from \eqref{eq:continuous-SOS-1}
it follows that
\[
\begin{split}
&
\lim_{t\to s}\operatorname{osc}(c\circ\alpha,[t/2,t])=\lim_{r\to s}\operatorname{osc}(c,[(\lambda/2)r,r])=0,
\\
&
\lim_{t\to s}|c(t)-c(\alpha(t))|=\lim_{r\to s}\operatorname{osc}(c,[\lambda r,r])=0
\end{split}
\]
for $s\in\{0,\infty\}$.
\end{proof}
Let $\beta:=\alpha_{-1}$ be the inverse function to $\alpha$.
\begin{lem}\label{le:SOS-inverse}
If $\alpha\in SOS(\mR_+)$, then $\beta\in SOS(\mR_+)$.
\end{lem}
\begin{proof}
Since $\log\alpha^\prime\in C_b(\mR_+)$ and $\alpha'\in
SO(\mR_+)$, we infer from the relations
\[
\beta'(t)=\frac{1}{\alpha'(\beta(t))},
\quad
|\beta'(t)-\beta'(\tau)|=
\frac{|\alpha^\prime(\beta(t))-\alpha^\prime(\beta(\tau))|}
{\alpha^\prime(\beta(t))\alpha^\prime(\beta(\tau))}
\quad
(t,\tau\in\mR_+)
\]
that $\log \beta'\in C_b(\mR_+)$ and $\beta'\in SO(\mR_+)$ too.
Thus $\beta$ belongs to $SOS(\mR_+)$.
\end{proof}
\section{Invertibility of binomial functional operators}\label{sec:FO-binomial}
\subsection{The case of the unit interval}
Let $C_b(\mI)$ denote the set of all bounded continuous
functions on $\mI:=(0,1)$. According to \cite{S77}, a function $\varphi\in
C_b(\mI)$ is called \textit{slowly oscillating at} $0$ if
\[
\lim_{r\to 0}\operatorname{osc}(\varphi,[\lambda r,r])=0
\]
for every (equivalently, some) $\lambda\in\mI$. A function $\varphi\in
C_b(\mI)$ is called slowly oscillating at $1$ if the function
$y\mapsto \varphi(1-y)$ slowly oscillates at $0$. Let $SO(\mI)$
denote the set of all functions in $C_b(\mI)$ that slowly
oscillate at $0$ and $1$. In this subsection we assume that $\alpha$
is an orientation-preserving diffeomorphism of $\mI$ onto itself that
has only two fixed points $0$ and $1$.
According to \cite{KKL03}, we say that $\alpha$ is a {\it
slowly oscillating shift} if $\log\alpha'\in C_b(\mI)$ and
$\alpha'\in SO(\mI)$. In the latter case we will write $\alpha\in SOS(\mI)$.
The shift operator $W_\alpha$ on the space $L^p(\mI)$
is defined by $W_\alpha f=f\circ\alpha$. It is easy to see that
$W_\alpha,W_\alpha^{-1}\in\cB(L^p(\mI))$ whenever $\alpha\in SOS(\mI)$.

From Theorem~1.2, Lemma~2.2, and Proposition~5.2 of \cite{KKL03} we extract
the following.
\begin{thm}\label{th:FO-interval}
Suppose $a,b\in SO(\mI)$ and $\alpha\in SOS(\mI)$. The
functional operator $aI-bW_\alpha$ is invertible on the
Lebesgue space $L^p(\mI)$ if and only if either
\begin{equation}\label{eq:FO-interval-1}
\inf_{y\in\mI}|a(y)|>0, \
\liminf_{y\to s}
\left( |a(y)|-|b(y)|\big(\alpha'(y)\big)^{-1/p}\right)>0 \
(s\in\{0,1\});
\end{equation}
or
\begin{equation}\label{eq:FO-interval-2}
\inf_{y\in\mI}|b(y)|>0, \
\limsup_{y\to s}
\left(|a(y)|-|b(y)|\big(\alpha'(y)\big)^{-1/p}\right)<0 \
(s\in\{0,1\}).
\end{equation}
If \eqref{eq:FO-interval-1} holds, then $(aI-bW_\alpha)^{-1}$ is given by \eqref{eq:FO-3}.
If \eqref{eq:FO-interval-2} is fulfilled, then $(aI-bW_\alpha)^{-1}$ is given by \eqref{eq:FO-4}.
\end{thm}
\subsection{Transplantation from the half-line to the unit interval}
Let $\eta:[0,1]\to\overline{\mR}_+=[0,+\infty]$ be defined by $\eta(y)=y/(1-y)$.
Then its inverse is given by $\eta^{-1}(t)=t/(1+t)$. Consider the isometric
isomorphism $G:L^p(\mR_+)\to L^p(\mI)$ defined by
\[
(G\varphi)(y):=(1-y)^{-2/p}\varphi[\eta(y)]\quad (y\in\mI).
\]
Its inverse is given by
\[
(G^{-1}f)(t):=(1+t)^{-2/p}f[\eta^{-1}(t)]\quad (t\in\mR_+).
\]
Let $I_\mI$ be the identity operator on $L^p(\mI)$ and
\[
(S_\mI\varphi)(x):=\frac{1}{\pi i}\int_\mI\frac{\varphi(y)}{y-x}dy\quad (x\in\mI).
\]
It is well known that the operator $S_\mI$ is bounded on $L^p(\mI)$.
\begin{lem}\label{le:transplantation}
Suppose $1<p<\infty$.
\begin{enumerate}
\item[{\rm(a)}] We have $GSG^{-1}=w_p^{-1}S_\mI w_pI_\mI$,
where $w_p(y):=(1-y)^{2/p-1}$ for $y\in\mI$.

\item[{\rm(b)}] If $a\in L^\infty(\mR_+)$, then
$G(aI)G^{-1}=(a\circ\eta)I_\mI$.

\item[{\rm (c)}] If $\alpha:\mR_+\to\mR_+$ is a diffeomorphism such that
$\log\alpha'\in L^\infty(\mR_+)$, then
$GW_\alpha G^{-1}=c_{\alpha,p} W_{\widetilde{\alpha}}$,
where
\[
\widetilde{\alpha}:=\eta^{-1}\circ\alpha\circ\eta,
\quad
c_{\alpha,p}(y):=\left(\frac{1-\widetilde{\alpha}(y)}{1-y}\right)^{2/p}
\quad\mbox{for}\quad y\in\mI.
\]
\end{enumerate}
\end{lem}
The proof of this lemma is straightforward and therefore it is omitted.
\begin{lem}\label{le:transplanted-functions}
Let $1<p<\infty$.
\begin{enumerate}
\item[{\rm(a)}]
If $a\in SO(\mR_+)$, then $a\circ\eta\in SO(\mI)$.

\item[{\rm(b)}]
If $\alpha\in SOS(\mR_+)$, then
$\widetilde{\alpha}\in SOS(\mI)$, $c_{\alpha,p}\in SO(\mI)$, and
\begin{equation}\label{eq:transplanted-functions-1}
0<\inf_{y\in\mI}c_{\alpha,p}(y)\le\sup_{y\in\mI}c_{\alpha,p}(y)<+\infty.
\end{equation}
\end{enumerate}
\end{lem}
\begin{proof}
(a) Let $\psi(y):=1-y$ for $y\in\mI$.
If $y\in(0,1/2]$, then $y\le\eta(y)\le 2y$ and $1/(2y)\le\eta(\psi(y))\le 1/y$.
Hence for $\lambda\in(0,1)$ and $r\in(0,1/2]$,
\[
\begin{split}
&
[\eta(\lambda r),\eta(r)]\subset [\lambda r,2r]=[(2^{-1}\lambda)2r,2r],
\\
&
[(\eta\circ\psi)(r),(\eta\circ\psi)(\lambda r)]\subset[(2r)^{-1},(\lambda r)^{-1}]
=
[2^{-1}\lambda(\lambda r)^{-1},(\lambda r)^{-1}].
\end{split}
\]
Therefore,
\begin{align}
\operatorname{osc}(a\circ\eta,[\lambda r,r])
&=
\operatorname{osc}(a,[\eta(\lambda r),\eta(r)])
\le
\operatorname{osc}(a,[(2^{-1}\lambda)2r,2r]),
\label{eq:transplanted-functions-2}
\\
\operatorname{osc}(a\circ\eta\circ\psi,[\lambda r,r])
&=
\operatorname{osc}(a,[(\eta\circ\psi)(r),(\eta\circ\psi)(\lambda r)])
\nonumber
\\
&\le
\operatorname{osc}(a,[2^{-1}\lambda(\lambda r)^{-1},(\lambda r)^{-1}]).
\label{eq:transplanted-functions-3}
\end{align}
Since $a\in SO(\mR_+)$, we get
\[
\lim_{r\to 0}
\operatorname{osc}(a,[(2^{-1}\lambda)2r,2r])
=
\lim_{r\to 0}
\operatorname{osc}(a,[2^{-1}\lambda(\lambda r)^{-1},(\lambda r)^{-1}])
=0.
\]
These equalities and inequalities \eqref{eq:transplanted-functions-2}--\eqref{eq:transplanted-functions-3}
imply that $a\circ\eta$ and $a\circ\eta\circ\psi$ slowly oscillate at zero.
Thus $a\circ\eta\in SO(\mI)$. Part (a) is proved.

(b) Let $y\in\mI$. Since $1-\widetilde{\alpha}(y)=1/[1+(\alpha\circ\eta)(y)]$, we have
\begin{equation}\label{eq:transplanted-functions-4}
\widetilde{\alpha}'(y)
=
\frac{(\alpha'\circ\eta)(y)}{(1-y)^2[1+(\alpha\circ\eta)(y)]^2}
=
(\alpha'\circ\eta)(y)c^2(y),
\end{equation}
where
\[
c(y):=\frac{1-\widetilde{\alpha}(y)}{1-y}>0.
\]
By Lemma~\ref{le:exp-repr}, $\alpha(t)=te^{\omega(t)}$ with $\omega\in SO(\mR_+)$.
Hence, for $t\in\mR_+$,
\[
(c\circ\eta^{-1})(t)
=
\frac{1+t}{1+te^{\omega(t)}}=
e^{-\omega(t)}\frac{1+t^{-1}}{1+t^{-1}e^{-\omega(t)}}.
\]
Since $SO(\mR_+)$ is a $C^*$-algebra, $e^{-\omega}\in SO(\mR_+)\subset C_b(\mR_+)$.
Therefore
\[
\lim_{t\to+\infty}
\frac{1+t^{-1}}{1+t^{-1}e^{-\omega(t)}}=1.
\]
This implies that the function $c\circ\eta^{-1}$ slowly oscillates at $+\infty$.
On the other hand, $e^\omega\in SO(\mR_+)\subset C_b(\mR_+)$. Then
\[
\lim_{t\to 0}(c\circ\eta^{-1})(t)=\lim_{t\to 0}\frac{1+t}{1+te^{\omega(t)}}=1.
\]
In particular, this implies that $c\circ\eta^{-1}$ slowly oscillates at zero.
Thus $c\circ\eta^{-1}$ belongs to $SO(\mR_+)$. By part (a) of this lemma, $c\in SO(\mI)$.
Similarly it can be shown that $1/c\in SO(\mI)$.

Since $SO(\mI)$ is a $C^*$-algebra, we conclude that $c_{\alpha,p}=c^{2/p}\in SO(\mI)$
and $c^2\in SO(\mI)$. By definition of a slowly oscillating shift, $\alpha'\in SO(\mR_+)$.
Then from part (a) we deduce that $\alpha'\circ\eta\in SO(\mI)\subset C_b(\mI)$.
Combining this observation with $c^2\in SO(\mI)$ and \eqref{eq:transplanted-functions-4},
we conclude that $\widetilde{\alpha}'\in SO(\mI)$.

We have already known that $c,1/c\in SO(\mI)\subset C_b(\mI)$. Hence \eqref{eq:transplanted-functions-1}
holds and $\log(c^2)\in C_b(\mI)$. On the other hand, $\alpha\in SOS(\mR_+)$.
Then $\log\alpha'\in C_b(\mR_+)$ and thus $\log(\alpha'\circ\eta)\in C_b(\mI)$.
Taking into account \eqref{eq:transplanted-functions-4}, we get
\[
\log(\widetilde{\alpha}')=\log(\alpha'\circ\eta)+\log(c^2)\in C_b(\mI),
\]
which concludes the proof of $\widetilde{\alpha}\in SOS(\mI)$.
\end{proof}
\subsection{Proof of Theorem~\ref{th:FO}}\label{sec:proof-FO}
From Lemma~\ref{le:transplantation}(b)-(c) it follows that
\begin{equation}\label{eq:proof-FO-1}
G(aI-bW_\alpha)G^{-1}=(a\circ\eta)I_\mI-(b\circ\eta)c_{\alpha,p}W_{\widetilde{\alpha}}.
\end{equation}
From Lemma~\ref{le:transplanted-functions} we know that $a\circ\eta\in SO(\mI)$,
$(b\circ\eta)c_{\alpha,p}\in SO(\mI)$, $\widetilde{\alpha}\in SOS(\mI)$, and
\[
0<C_1:=\inf_{y\in\mI}c_{\alpha,p}(y)\le\sup_{y\in\mI}c_{\alpha,p}(y)=:C_2<+\infty.
\]
It is easy to see that
\begin{equation}\label{eq:proof-FO-2}
\inf_{y\in\mI}|(a\circ\eta)(y)|=\inf_{t\in\mR_+}|a(t)|,
\quad
\inf_{y\in\mI}|(b\circ\eta)(y)|=\inf_{t\in\mR_+}|b(t)|.
\end{equation}
Hence
\begin{equation}\label{eq:proof-FO-3}
C_1\inf_{t\in\mR_+}|b(t)|
\le\inf_{y\in\mI}|(b\circ\eta)(y)c_{\alpha,p}(y)|\le
C_2\inf_{t\in\mR_+}|b(t)|.
\end{equation}
Further, it can be checked straightforwardly that for every $y\in\mI$,
\[
\begin{split}
&
|(a\circ\eta)(y)|-|(b\circ\eta)(y)c_{\alpha,p}(y)|\big(\widetilde{\alpha}'(y)\big)^{-1/p}
\\
&=
|(a\circ\eta)(y)|-|(b\circ\eta)(y)|\big((\alpha'\circ\eta)(y)\big)^{-1/p}.
\end{split}
\]
Therefore
\begin{align}
\limsup_{y\to 0}\Big/\liminf_{y\to 0}
&\Big(|(a\circ\eta)(y)|-|(b\circ\eta)(y)c_{\alpha,p}(y)|\big(\widetilde{\alpha}'(y)\big)^{-1/p}\Big)
\nonumber
\\
&=
\limsup_{t\to 0}\Big/\liminf_{t\to 0}
\Big(|a(t)|-|b(t)|\big(\alpha'(t)\big)^{-1/p}\Big),
\label{eq:proof-FO-4}
\\
\limsup_{y\to 1}\Big/\liminf_{y\to 1}
&\Big(|(a\circ\eta)(y)|-|(b\circ\eta)(y)c_{\alpha,p}(y)|\big(\widetilde{\alpha}'(y)\big)^{-1/p}\Big)
\nonumber
\\
&=
\limsup_{t\to\infty}\Big/\liminf_{t\to\infty}
\Big(|a(t)|-|b(t)|\big(\alpha'(t)\big)^{-1/p}\Big),
\label{eq:proof-FO-5}
\end{align}
respectively. Equality \eqref{eq:proof-FO-1} says that $aI-bW_\alpha$
is invertible on $L^p(\mR_+)$ if and only if
$(a\circ\eta)I_\mI-(b\circ\eta)c_{\alpha,p}W_{\widetilde{\alpha}}$
is invertible on $L^p(\mI)$. On the other hand, equality
\eqref{eq:proof-FO-2}, inequalities \eqref{eq:proof-FO-3}, and equalities
\eqref{eq:proof-FO-4}--\eqref{eq:proof-FO-5} imply that the conditions
of Theorems~\ref{th:FO} and~\ref{th:FO-interval} are equivalent.

Further, if \eqref{eq:FO-1} holds, then \eqref{eq:FO-interval-1} is
fulfilled. Then, by Theorem~\ref{th:FO-interval},
\begin{equation}\label{eq:proof-FO-6}
\big((a\circ\eta)I_\mI-(b\circ\eta)c_{\alpha,p}W_{\widetilde{\alpha}}\big)^{-1}
=
\sum_{n=0}^\infty
\left(\frac{b\circ\eta}{a\circ\eta}c_{\alpha,p}W_{\widetilde{\alpha}}\right)^n
\frac{1}{a\circ\eta}I_\mI.
\end{equation}
Applying Lemma~\ref{le:transplantation}(b)--(c), we obtain
\begin{equation}\label{eq:proof-FO-7}
\sum_{n=0}^\infty
\left(\frac{b\circ\eta}{a\circ\eta}c_{\alpha,p}W_{\widetilde{\alpha}}\right)^n
\frac{1}{a\circ\eta}I_\mI=G\left(\sum_{n=0}^\infty (a^{-1}bW_\alpha)^na^{-1}I\right)G^{-1}.
\end{equation}
Combining \eqref{eq:proof-FO-1}, \eqref{eq:proof-FO-6}, and \eqref{eq:proof-FO-7}, we get
\eqref{eq:FO-3}.

Analogously it can be shown that if \eqref{eq:FO-2} is fulfilled, then
$(aI-bW_\alpha)^{-1}$ is calculated by \eqref{eq:FO-4}.
\qed
\section{Convolution operators}\label{sec:Mellin-convolution}
\subsection{Fourier convolution operators}
Let $F:L^2(\mR)\to L^2(\mR)$ denote the Fourier transform,
\[
(Ff)(x):=\int_\mR f(y)e^{-ixy}dy\quad (x\in\mR),
\]
and let $F^{-1}:L^2(\mR)\to L^2(\mR)$ be the inverse of $F$. A function
$a\in L^\infty(\mR)$ is called a Fourier multiplier on $L^p(\mR)$
if the mapping
$f\mapsto F^{-1}aFf$ maps $L^2(\mR)\cap L^p(\mR)$ onto itself and extends
to a bounded operator on $L^p(\mR)$. The latter operator is then denoted by
$W^0(a)$. We let $M_p(\mR)$ stand for the set of all Fourier multipliers on
$L^p(\mR)$. One can show that $M_p(\mR)$ is a Banach algebra under the norm
\[
\|a\|_{M_p(\mR)}:=\|W^0(a)\|_{\cB(L^p(\mR))}.
\]
We denote by $PC$ the $C^*$-algebra of all bounded piecewise continuous
functions on $\dot{\mR}=\mR\cup\{\infty\}$. By definition, $a\in PC$ if
and only if $a\in L^\infty(\mR)$ and the one-sided limits
\[
a(x_0-0):=\lim_{x\to x_0-0}a(x),
\quad
a(x_0+0):=\lim_{x\to x_0+0}a(x)
\]
exist for each $x_0\in\dot{\mR}$. If a function $a$ is given everywhere
on $\mR$, then its total variation is defined by
\[
V(a):=\sup\sum_{k=1}^n|a(x_k)-a(x_{k-1})|,
\]
where the supremum is taken over all $n\in\mN$ and
\[
-\infty<x_0<x_1<\dots<x_n<+\infty.
\]
If $a$ has a finite total variation, then it has finite one-sided limits
$a(x-0)$ and $a(x+0)$ for all $x\in\dot{\mR}$, that is, $a\in PC$
(see, e.g., \cite[Chap. VIII, Sections~3 and~9]{N55}). The following
theorem gives an important subset of $M_p(\mR)$. Its proof can be found,
e.g., in \cite[Theorem~17.1]{BKS02}.
\begin{thm}[Stechkin's inequality]
\label{th:Stechkin}
If $a\in PC$ has finite total variation $V(a)$, then $a\in M_p(\mR)$ and
\[
\|a\|_{M_p(\mR)}\le\|S_\mR\|_{\cB(L^p(\mR))}\big(\|a\|_{L^\infty(\mR)}+V(a)\big),
\]
where $S_\mR$ is the Cauchy singular integral operator on $\mR$.
\end{thm}
According to \cite[p.~325]{BKS02}, let $PC_p$ be the closure in $M_p(\mR)$
of the set of all functions $a\in PC$ with finite total variation on $\mR$.
Following \cite[p.~331]{BKS02}, put $C_p(\overline{\mR}):=PC_p\cap C(\mR)$,
where $\overline{\mR}:=[-\infty,+\infty]$.
\subsection{Mellin convolution operators}
Let $d\mu(t)=dt/t$ be the (normalized) invariant measure on $\mR_+$.
Consider the Fourier transform on $L^2(\mathbb{R}_+,d\mu)$, which is
usually referred to as the Mellin transform and is defined by
\[
M:L^2(\mR_+,d\mu)\to L^2(\mR),
\quad
(Mf)(x)=\int_{\mR_+} f(t) t^{-ix}\,\frac{dt}{t}.
\]
It is an invertible operator, with inverse given by
\[
{M^{-1}}:L^2(\mR)\to L^2(\mR_{+},d\mu),
\quad
({M^{-1}}g)(t)= \frac{1}{2\pi}\int_{\mR}
g(x)t^{ix}\,dx.
\]
Let $E$ be the isometric isomorphism
\begin{equation}\label{eq:def-E}
E:L^p(\mR_+,d\mu)\to L^p(\mR),
\quad
(Ef)(x):=f(e^x)\quad (x\in\mR).
\end{equation}
Then the map $A\mapsto E^{-1}AE$ transforms the Fourier convolution
operator $W^0(a)=F^{-1}aF$ to the Mellin convolution operator
\[
\operatorname{Co}(a):=M^{-1}aM
\]
with the same symbol $a$. Hence the class of Fourier multipliers on
$L^p(\mR)$ coincides with the class of Mellin multipliers on $L^p(\mR_+,d\mu)$.

The following result was obtained in \cite[Proposition~1.6]{D87},
its proof can also be found in \cite[Proposition~12.7]{RS90}.
\begin{thm}[Duduchava]
\label{th:Duduchava}
If $a\in C_b(\mR_+)$ and $b\in PC_p$ are such that
\[
\lim_{t\to 0+0}a(t)=\lim_{t\to+\infty}a(t)=\lim_{x\to-\infty}b(x)=\lim_{x\to +\infty}b(x)=0,
\]
then $a\operatorname{Co}(b)\in\cK(L^p(\mR_+,d\mu))$.
\end{thm}
\subsection{The algebra $\cA$}
Let $\mathfrak{A}$ be a Banach algebra and $\mathfrak{S}$ be a subset of
$\mathfrak{A}$. By $\clos_\mathfrak{A}\mathfrak{S}$ we denote
the closure of $\mathfrak{S}$ in the norm of $\mathfrak{A}$. Following
\cite[Section~3.45]{BS06}, we denote by $\alg_\mathfrak{A}\mathfrak{S}$
the smallest closed subalgebra of $\mathfrak{A}$ containing $\mathfrak{S}$
and by $\operatorname{id}_\mathfrak{A}\mathfrak{S}$ the smallest closed
two-sided ideal of $\mathfrak{A}$ containing $\mathfrak{S}$.

Let $1<p<\infty$. Put
\[
\cB:=\cB(L^p(\mR_+)),\quad
\cK:=\cK(L^p(\mR_+)),\quad
\cA:=\alg_\cB\{I,S\}.
\]
Obviously, the algebra $\cA$ is commutative.
For $\beta\in\mC$, let
\[
(R_\beta f)(t):=\frac{1}{\pi i}\int_0^\infty\frac{f(\tau)}{\tau-e^{i\beta}t}\quad (t\in\mR_+)
\]
and write $R$ for $R_\pi$. Further, put
\[
s_p(x):=\coth[\pi(x+i/p)],
\quad
r_{p,\beta}(x):=\frac{e^{(x+i/p)(\pi-\beta)}}{\sinh[\pi(x+i/p)]}
\quad(x\in\mR)
\]
and write $r_p$ for $r_{p,\pi}$. Consider the isometric isomorphism
\begin{equation}\label{eq:def-Phi}
\Phi:L^p(\mR_+)\to L^p(\mR_+,d\mu),
\quad
(\Phi f)(t):=t^{1/p}f(t)\quad(t\in\mR_+).
\end{equation}
The following facts are well known. Their proofs can be found, e.g.,
in \cite[Propositions~2.1--2.5]{RS90} (see also \cite{D87}).
\begin{thm}\label{th:algebra-A}
Suppose $1<p<\infty$.
\begin{enumerate}
\item[{\rm(a)}]
The algebra $\cA$ is the smallest closed subalgebra of $\cB$ that contains
the operators $\Phi^{-1}\operatorname{Co}(a)\Phi$ with $a\in C_p(\overline{\mR})$.

\item[{\rm(b)}]
If $\beta\in\mC$ and $\operatorname{Re}\beta\in(0,2\pi)$, then
$s_p,r_{p,\beta}\in C_p(\overline{\mR})$ and
\[
\Phi S\Phi^{-1}=\operatorname{Co}(s_p),
\quad
\Phi R_\beta\Phi^{-1}=\operatorname{Co}(r_{p,\beta}).
\]

\item[{\rm(c)}]
The maximal ideal space of the commutative Banach algebra $\cA$ is homeomorphic
to $\overline{\mR}$. In particular, an operator $\Phi^{-1}\operatorname{Co}(a)\Phi$
with $a\in C_p(\overline{\mR})$ is invertible if and only if $a(x)\ne 0$ for all
$x\in\overline{\mR}$. Thus $\cA$ is an inverse closed subalgebra of $\cB$.

\item[{\rm(d)}]
An operator $\Phi^{-1}\operatorname{Co}(a)\Phi$ with $a\in C_p(\overline{\mR})$
belongs to $\operatorname{id}_\cA\{R\}$ if and only if $a(-\infty)=a(+\infty)=0$.
\end{enumerate}
\end{thm}
From $s_p^2-r_p^2=1$ and Theorem~\ref{th:algebra-A}(b) it follows that
\begin{equation}\label{eq:S-R-relation}
4P_+P_-=4P_-P_+=I-S^2=-R^2.
\end{equation}

Let us describe the quotient algebra $\cA^\pi:=(\cA+\cK)/\cK$. Since a Mellin
convolution operator is Fredholm if and only if it is invertible, from
Theorem~\ref{th:algebra-A} we obtain the following.
\begin{cor}\label{co:algebra-Api}
\begin{enumerate}
\item[{\rm(a)}]
The algebra $\cA^\pi$ is commutative and its maximal ideal space is homeomorphic
to $\overline{\mR}$.

\item[{\rm(b)}]
The Gelfand transform of a coset $(\Phi^{-1}\operatorname{Co}(a)\Phi)^\pi\in\cA^\pi$
for $a\in C_p(\overline{\mR})$ is given by
\[
\big[(\Phi^{-1}\operatorname{Co}(a)\Phi)^\pi\big]
\widehat{\hspace{2mm}}(x)=a(x)\quad\mbox{for}\quad x\in\overline{\mR}.
\]
In particular,
\[
(S^\pi)\widehat{\hspace{2mm}}(\pm\infty)=\pm 1,
\quad
(S^\pi)\widehat{\hspace{2mm}}(x)=s_p(x)
\quad\mbox{for}\quad x\in\mR,
\]
and if $\beta\in\mC$ and $\operatorname{Re}\beta\in(0,2\pi)$, then
\[
(R_\beta^\pi)\widehat{\hspace{2mm}}(\pm\infty)=0,
\quad
(R_\beta^\pi)\widehat{\hspace{2mm}}(x)=r_{p,\beta}(x)
\quad\mbox{for}\quad x\in\mR.
\]

\item[{\rm(c)}]
An operator $H\in\cA$ belongs to $\operatorname{id}_\cA\{R\}$ if and
only if
\[
(H^\pi)\widehat{\hspace{2mm}}(-\infty)=(H^\pi)\widehat{\hspace{2mm}}(+\infty)=0.
\]
\end{enumerate}
\end{cor}
\section{Mellin pseudodifferential operators}\label{sec:Mellin-PDO}
\subsection{Boundedness}
If $a$ is an absolutely continuous function of finite total variation on $\mR$,
then $a'\in L^1(\mR)$ and
\[
V(a)=\int_\mR|a'(x)|dx
\]
(see, e.g., \cite[Chap. VIII, Sections 3 and 9; Chap. XI, Section~4]{N55}).
The set $V(\mR)$ of all absolutely continuous functions of finite total variation
on $\mR$ forms a Banach algebra when equipped with the norm
\[
\|a\|_V:=\|a\|_{L^\infty(\mR)}+V(a)=\|a\|_{L^\infty(\mR)}+\int_\mR|a'(x)|dx.
\]

Following \cite{K06,K06-IWOTA}, let $C_b(\mR_+,V(\mR))$ denote the Banach
algebra of all bounded continuous $V(\mR)$-valued functions on $\mR_+$
with the norm
\[
\|\fa(\cdot,\cdot)\|_{C_b(\mR_+,V(\mR))}=\sup_{t\in\mR_+}\|\fa(t,\cdot)\|_V.
\]
As usual, let $C_0^\infty(\mR_+)$ be the set of all infinitely differentiable
functions of compact support on $\mR_+$.

The following boundedness
result for Mellin pseudodifferential operators was obtained
in \cite[Theorem~3.1]{K06-IWOTA} (see also \cite[Theorem~3.1]{K06}).
\begin{thm}\label{th:boundedness-PDO}
If $\fa\in C_b(\mR_+,V(\mR))$, then the Mellin pseudodifferential operator
$\operatorname{Op}(\fa)$, defined for functions $f\in C_0^\infty(\mR_+)$ by the iterated
integral
\[
[\operatorname{Op}(\fa)f](t)=\frac{1}{2\pi}\int_\mR dx \int_{\mR_+}
\fa(t,x)\left(\frac{t}{\tau}\right)^{ix}f(\tau) \frac{d\tau}{\tau}
\quad\mbox{for}\quad t\in\mR_+,
\]
extends to a bounded linear operator on the space $L^p(\mR_+,d\mu)$ and there is a
number $C_p\in(0,\infty)$ depending only on $p$ such that
\[
\|\operatorname{Op}(\fa)\|_{\cB(L^p(\mR_+,d\mu))}
\le
C_p
\|\fa\|_{C_b(\mR_+,V(\mR))}.
\]
\end{thm}
\subsection{Compactness of commutators}
Let $SO(\mR_+,V(\mR))$ denote the Banach subalgebra of
$C_b(\mR_+,V(\mR))$ consisting of all $V(\mR)$-valued functions
$\fa$ on $\mR_+$ that slowly oscillate at $0$ and
$\infty$, that is,
\[
\lim_{r\to 0} \operatorname{cm}_r^C(\fa)=
\lim_{r\to \infty} \operatorname{cm}_r^C(\fa)=0,
\]
where
\[
\operatorname{cm}_r^C(\fa)=
\max\big\{\big\|\fa(t,\cdot)-\fa(\tau,\cdot)\big\|_{L^\infty(\mR)}:
t,\tau\in[r,2r]\big\}.
\]
Let $\cE(\mR_+,V(\mR))$ be the Banach algebra of all $V(\mR)$-valued functions
$\fa$ belonging to $SO(\mR_+,V(\mR))$ and such that
\begin{equation}\label{eq:SO-V}
\lim_{|h|\to 0}\sup_{t\in\mR_+}\big\|\fa(t,\cdot)-\fa^h(t,\cdot)\big\|_V=0
\end{equation}
where $\fa^h(t,x):=\fa(t,x+h)$ for all $(t,x)\in
\mR_+\times \mR$.

The following result on compactness of commutators of Mellin pseudodifferential
operators was obtained in \cite[Theorem~3.5]{K08} (see also \cite[Corollary~8.4]{K06}).
\begin{thm}\label{th:comp-commutators-PDO}
If $\fa,\fb\in\cE(\mR_+,V(\mR))$, then the commutator
\[
\operatorname{Op}(\fa)\operatorname{Op}(\fb)-\operatorname{Op}(\fb)\operatorname{Op}(\fa)
\]
is compact on the space $L^p(\mR_+,d\mu)$.
\end{thm}
\section{Localization}\label{sec:localization}
\subsection{The Allan-Douglas local principle}
The Allan-Douglas local principle is one of the main tools in
studying singular integral operators in the last decades.
The aim of this section is to apply this principle to
operators in the algebra
\[
\cF:=\alg_\cB\big\{aI,S,W_\alpha,W_\alpha^{-1}:a\in SO(\mR_+)\big\}.
\]
Here is the formulation of the local principle taken from
\cite[Theorem~1.35(a)]{BS06}.
Let $\fA$ be a Banach algebra with identity. A subalgebra $\fZ$ of
$\fA$ is said to be a \textit{central subalgebra} if $za=az$ for
all $z\in\fZ$ and all $a\in\fA$.
\begin{thm}[Allan-Douglas]\label{th:AllanDouglas}
Let $\fA$ be a Banach algebra with identity
$e$ and let $\fZ$ be a closed central subalgebra of $\fA$ containing
$e$. Let $M(\fZ)$ be the maximal ideal space of $\fZ$, and for
$\omega\in M(\fZ)$, let $\fJ_\omega$ refer to the smallest closed
two-sided ideal of $\fA$ containing the ideal $\omega$. Then an
element $a$ is invertible in $\fA$ if and only if $a+\fJ_\omega$ is
invertible in the quotient algebra $\fA/\fJ_\omega$ for all $\omega\in
M(\fZ)$.
\end{thm}
The algebra $\fA/\fJ_\omega$ is referred to as the \textit{local
algebra} of $\fA$ at $\omega\in M(\fZ)$.

Now we are going to construct an algebra $\Lambda$ that contains the algebra
$\cF$ and such that the quotient algebra $\Lambda^\pi:=\Lambda/\cK$
has a center properly containing $\cA^\pi$. To this end we need several
compactness results.
\subsection{Compactness results}
The following compactness results were obtained in \cite[Corollaries 5.2--5.3]{KK03}.
\begin{thm}\label{th:compactness-commutators-interval}
Let $1<p<\infty$.
\begin{enumerate}
\item[{\rm(a)}]
If $a\in SO(\mI)$, then $aS_\mI-S_\mI aI_\mI\in\cK(L^p(\mI))$.

\item[{\rm(b)}]
If $\alpha\in SOS(\mI)$, then $W_\alpha S_\mI-S_\mI W_\alpha\in\cK(L^p(\mI))$.
\end{enumerate}
\end{thm}
Now we prove their counterparts for the case of $\mR_+$.
\begin{thm}\label{th:comp-commutators}
Let $1<p<\infty$.
\begin{enumerate}
\item[{\rm(a)}]
If $a\in SO(\mR_+)$, then $aS-SaI\in\cK$.

\item[{\rm(b)}]
If $\alpha\in SOS(\mR_+)$, then $W_\alpha S-SW_\alpha\in\cK$.
\end{enumerate}
\end{thm}
\begin{proof}
(a)
If $a\in SO(\mR_+)$, then $a\circ\eta\in SO(\mI)$ by Lemma~\ref{le:transplanted-functions}(a).
Then in view of Theorem~\ref{th:compactness-commutators-interval}(a),
the operator $(a\circ\eta)S_\mI-S_\mI(a\circ\eta)I_\mI$ is compact on
$L^2(\mI)$. For $p=2$ from Lemma~\ref{le:transplantation}(a),(b) it follows that
\[
aS-SaI=G^{-1}[(a\circ\eta) S_\mI-S_\mI(a\circ\eta)I_\mI]G\in\cK(L^2(\mR_+)).
\]
Since the operator $aS-SaI$ is bounded on all spaces $L^p(\mR_+)$, $p\in(1,\infty)$,
from the Krasnosel'skii interpolation theorem (see
\cite[Theorem~3.10]{KZPS76}) it follows that the operator $aS-SaI$ is compact
on all spaces $L^p(\mR_+)$, $p\in(1,\infty)$.

(b)
If $\alpha\in SOS(\mR_+)$, then $\widetilde{\alpha}\in SOS(\mR_+)$
and $c_{\alpha,2}\in SO(\mI)$ due to Lemma~\ref{le:transplanted-functions}(b).
From Theorem~\ref{th:compactness-commutators-interval} it follows that
\[
K_1:=W_{\widetilde{\alpha}}S_\mI-S_\mI W_{\widetilde{\alpha}}\in\cK(L^2(\mI)),
\quad
K_2:=c_{\alpha,2}S_\mI-S_\mI c_{\alpha,2}I_\mI\in\cK(L^2(\mI)).
\]
Then
\[
c_{\alpha,2}W_{\widetilde{\alpha}}S_\mI-S_\mI c_{\alpha,2}W_{\widetilde{\alpha}}
=
c_{\alpha,2}K_1+K_2W_{\widetilde{\alpha}}
\in\cK(L^2(\mI)).
\]
From Lemma~\ref{le:transplantation}(a),(c) we get
\[
W_\alpha S-SW_\alpha=
G^{-1}[c_{\alpha,2}W_{\widetilde{\alpha}}S_\mI-S_\mI c_{\alpha,2}W_{\widetilde{\alpha}}]G
\in\cK(L^2(\mR_+)).
\]
Since the operator $W_\alpha S-SW_\alpha$ is bounded on all spaces
$L^p(\mR_+)$, $1<p<\infty$, we conclude that the
operator $W_\alpha S-SW_\alpha$ is compact on all spaces $L^p(\mR_+)$,
$1<p<\infty$, by analogy with part (a).
\end{proof}
\begin{cor}\label{co:comp-commutators-algebra}
If $a\in SO(\mR_+)$ and $\alpha\in SOS(\mR_+)$, then for every $A\in\cA$,
\[
aA-AaI\in\cK,
\quad
W_\alpha A-AW_\alpha\in\cK.
\]
\end{cor}
\begin{proof}
It is easy to see that if $B\in\cB$ is such that $BS-SB\in\cK$, then for
every $A\in\cA$, the commutator $BA-AB$ is compact. It remains to apply
Theorem~\ref{th:comp-commutators}.
\end{proof}
\begin{thm}\label{th:comp-aR}
If $a\in C_b(\mR_+)$ and $\lim\limits_{t\to s}a(t)=0$ for $s\in\{0,\infty\}$,
then $aR\in\cK$.
\end{thm}
\begin{proof}
By Theorem~\ref{th:algebra-A}(b), $r_p\in C_p(\overline{\mR})\subset PC_p$ and
$R=\Phi^{-1}\operatorname{Co}(r_p)\Phi$. It is easy to see that
\[
\lim_{x\to-\infty}r_p(x)=\lim_{x\to+\infty}r_p(x)=0.
\]
Hence, by Theorem~\ref{th:Duduchava}, $a\operatorname{Co}(r_p)\in\cK(L^p(\mR_+,d\mu))$.
Therefore the operator $aR=a\Phi^{-1}\operatorname{Co}(r_p)\Phi=\Phi^{-1}a\operatorname{Co}(r_p)\Phi$
is compact on $L^p(\mR_+)$.
\end{proof}
\begin{cor}\label{co:comp-aH}
If $a\in C_b(\mR_+)$, $\lim\limits_{t\to s}a(t)=0$ for $s\in\{0,\infty\}$ and
$H\in\operatorname{id}_\cA\{R\}$, then $aH\in\cK$.
\end{cor}
\begin{proof}
Since $\cA$ is commutative, we see that
$\operatorname{id}_\cA\{R\}=\clos_\cA\{RA:A\in\cA\}$.
From this equality and Theorem~\ref{th:comp-aR} we immediately get the statement.
\end{proof}
\subsection{Algebras $\cZ$, $\cD$, and $\Lambda$}
Let us consider
\[
\begin{split}
\cZ &:=\alg_\cB\big\{I,S,cR,K:\ c\in SO(\mR_+),\ K\in\cK\big\},
\\
\cD &:=\alg_\cB\big\{aI,S:\ a\in SO(\mR_+)\big\},
\\
\Lambda &:=\big\{A\in\cB:\ AC-CA\in\cK\text{ for all }C\in\cZ\big\}.
\end{split}
\]
\begin{lem}\label{le:alg-Lambda}
\begin{enumerate}
\item[{\rm (a)}]
The set $\Lambda$ is a closed unital subalgebra of $\cB$.

\item[{\rm (b)}]
The set $\cK$ is a closed two-sided ideal of the algebra $\Lambda$.

\item[{\rm (c)}]
An operator $A\in\Lambda$ is Fredholm if and only if the coset $A^\pi:=A+\cK$
is invertible in the quotient algebra $\Lambda^\pi:=\Lambda/\cK$.
\end{enumerate}
\end{lem}
The proof is straightforward and therefore it is omitted.
\begin{thm}\label{th:embeddings}
We have $\cK\subset\cZ\subset\cD\subset\cF\subset\Lambda$.
\end{thm}
\begin{proof}
The inclusion $\cK\subset\cZ$ follows from the definition of the algebra $\cZ$.
The inclusion $\cD\subset\cF$ is obvious.

To prove that $\cZ\subset\cD$, it is sufficient to show that all the
generators of $\cZ$ belong to $\cD$. Obviously, $I,S\in\cD$.
Further, $cR\in\cD$ for $c\in SO(\mR_+)$ because $R\in\cA\subset\cD$.
It is well known that $\cK\subset\alg_\cB\{aI,S:a\in C(\overline{\mR}_+)\}$
(see, e.g., \cite[Lemma~8.23]{BK97} for Carleson Jordan curves,
in the present case the proof is analogous). Since
$C(\overline{\mR}_+)\subset SO(\mR_+)$, from the above inclusion
it follows that $\cK\subset\cD$. Thus, $\cZ\subset\cD$.

Let us show $\cF\subset\Lambda$.
By Corollary~\ref{co:comp-commutators-algebra},
$aI,S\in\Lambda$ for $a\in SO(\mR_+)$ and
\begin{equation}\label{eq:D-Lambda-7}
W_\alpha S-SW_\alpha\in\cK,\quad W_\alpha R-RW_\alpha\in\cK.
\end{equation}
Since $\alpha\in SOS(\mR_+)$, from
Lemma~\ref{le:SOS-inverse} we infer that $\alpha_{-1}\in
SOS(\mR_+)$, too. From Lemma~\ref{le:continuous-SOS} and
Theorem~\ref{th:comp-aR} it follows that
\begin{equation}\label{eq:D-Lambda-8}
(c-c\circ\alpha_{-1})R\in\cK.
\end{equation}
Combining the second relation in \eqref{eq:D-Lambda-7} and
relation \eqref{eq:D-Lambda-8}, we get
\begin{equation}\label{eq:D-Lambda-9}
W_\alpha cR-cRW_\alpha
=
W_\alpha(c-c\circ\alpha_{-1})R+c(W_\alpha R-RW_\alpha)\in\cK.
\end{equation}
From \eqref{eq:D-Lambda-7} and
\eqref{eq:D-Lambda-9} it follows that $W_\alpha\in\Lambda$.

Taking into account that $W_\alpha^{-1}=W_{\alpha_{-1}}$ and
repeating the above argument with $\alpha_{-1}\in SOS(\mR_+)$ in
place of $\alpha$ we finally get $W_\alpha^{-1}\in\Lambda$. We
have proved that all the generators of $\cF$ belong to $\Lambda$.
Thus $\cF\subset\Lambda$.
\end{proof}
From the above theorem it follows that the quotient algebras $\cZ^\pi:=\cZ/\cK$,
$\cD^\pi:=\cD/\cK$, and $\Lambda^\pi:=\Lambda/\cK$ are well defined.
Clearly, $\cZ^\pi$ lies in the center of $\Lambda^\pi$. Our next
aim is to describe the maximal ideal space of $\cZ^\pi$. We start with
a description of the maximal ideal space of the bigger algebra $\cD^\pi$,
which is commutative in view of Theorem~\ref{th:comp-commutators}(a).
\subsection{Maximal ideal space of $\cD^\pi$}
\begin{thm}\label{th:maximal-ideal-space-Dpi}
The maximal ideal space $M(\cD^\pi)$ of the commutative Banach algebra
$\cD^\pi$ is homeomorphic to the set
\begin{equation}\label{eq:mis-D-1}
\mathfrak{M}:=\big(M(SO(\mR_+))\times\{-\infty,+\infty\}\big)\cup(\Delta\times\mR),
\end{equation}
where $\Delta$ is given by \eqref{eq:a1}.
\end{thm}
\begin{proof}
The proof is similar to that of \cite[Theorem~6.3]{BFK06}.
It is clear that $\cC^\pi:=\{(cI)^\pi:c\in SO(\mR_+)\}$
and $\cA^\pi$ are commutative Banach subalgebras of the algebra $\cD^\pi$
and their maximal ideal spaces can be identified with $M(SO(\mR_+))$ and
$\overline{\mR}$, respectively. Fix $(\xi,x)\in M(SO(\mR_+))\times\overline{\mR}$
and consider the maximal ideals $\{(cI)^\pi:c\in SO(\mR_+),c(\xi)=0\}$
of $\cC^\pi$ and $\{A^\pi\in\cA^\pi:(A^\pi)\widehat{\hspace{2mm}}(x)=0\}$ of
$\cA^\pi$. Let $\cN_{\xi,x}^\pi$ denote the closed two-sided
(not necessarily maximal) ideal of $\cD^\pi$ generated by the above
maximal ideals of $\cC^\pi$ and $\cA^\pi$. Identifying the pair
$(\xi,x)\in M(SO(\mR_+))\times\overline{\mR}$ with the ideal $\cN_{\xi,x}^\pi$
and taking into account \eqref{eq:mis-D-1} and $M(SO(\mR_+))=\mR_+\cup\Delta$,
we see that $M(\cD^\pi)$ is homeomorphic to a subset of $M(SO(\mR_+))\times\overline{\mR}
=\big(\mR_+\times\mR\big)\cup\mathfrak{M}$.

Observe that due to \eqref{eq:S-R-relation} any coset $D^\pi\in\cD^\pi$ is of the form
\begin{equation}\label{eq:mis-D-2}
D^\pi=(c_+P_+)^\pi+(c_-P_-)^\pi+Y^\pi,
\end{equation}
where $c_\pm\in SO(\mR_+)$, $P_\pm=(I\pm S)/2$, and
\begin{equation}\label{eq:mis-D-3}
Y^\pi=\lim_{n\to\infty}\sum_{k=1}^{m_n}(c_{n,k}H_{n,k})^\pi
\end{equation}
with $c_{n,k}\in SO(\mR_+)$, $H_{n,k}\in\operatorname{id}_\cA\{R\}$,
and $m_n\in\mN$.

Fix $(\xi,x)\in\mR_+\times\mR$. Given a coset $D^\pi$ of the form
\eqref{eq:mis-D-2}--\eqref{eq:mis-D-3}, we can choose continuous functions
$\widetilde{c}_\pm$, $\widetilde{c}_{n,k}$ with compact support on $\mR_+$ such that
$\widetilde{c}_\pm(\xi)=c_\pm(\xi)$, $\widetilde{c}_{n,k}(\xi)=c_{n,k}(\xi)$
and operators $\widetilde{H}_\pm,\widetilde{H}_{n,k}\in\operatorname{id}_\cA\{R\}$
such that $(\widetilde{H}_\pm^\pi)\widehat{\hspace{2mm}}(x)=(P_\pm^\pi)\widehat{\hspace{2mm}}(x)$,
$(\widetilde{H}_{n,k}^\pi)\widehat{\hspace{2mm}}(x)=(H_{n,k}^\pi)\widehat{\hspace{2mm}}(x)$.
Then
\[
(c_\pm P_\pm)^\pi=\big[(c_\pm-\widetilde{c}_\pm)P_\pm\big]^\pi+
[\widetilde{c}_\pm(P_\pm-\widetilde{H}_\pm)\big]^\pi+(\widetilde{c}_\pm
\widetilde{H}_\pm)^\pi,
\]
respectively, where the first two terms on the right-hand side belong
to the ideal $\cN^\pi_{\xi,x}$, whereas the terms $(\widetilde{c}_\pm
\widetilde{H}_\pm)^\pi$ are zero by Corollary~\ref{co:comp-aH}. Hence,
$(c_\pm P_\pm)^\pi\in\cN^\pi_{\xi,x}$. Analogously,
$(c_{n,k}H_{n,k})^\pi\in\cN^\pi_{\xi,x}$ for all pairs $(n,k)$. Thus,
$\cN^\pi_{\xi,x}$ for $(\xi,x)\in\mR_+\times\mR$ contains every coset
$D^\pi\in\cD^\pi$ and hence cannot be a maximal ideal. In fact,
it is not even a proper ideal, since it contains the unit $I^\pi$.
Consequently, $M(\cD^\pi)\subset\mathfrak{M}$.

Consider now a point $(\xi,x)\in\mathfrak{M}$. Then $\cN^\pi_{\xi,x}$ is a proper
ideal of $\cD^\pi$ because it does not contain the unit $I^\pi$. Let us
show that the ideal $\cN^\pi_{\xi,x}$ is maximal. Assume the contrary:
the ideal $\cN_{\xi,x}^\pi$ is not maximal. Then there is a maximal
ideal $\widetilde{\cN}_{\xi,x}^\pi$ of $\cD^\pi$ that contains properly
$\cN_{\xi,x}^\pi$. For any $D^\pi\in\cD^\pi$ of the form \eqref{eq:mis-D-2},
we have $D^\pi=(D^\pi)\widehat{\hspace{2mm}}(\xi,x)I^\pi+O_{\xi,x}^\pi$, where
$O_{\xi,x}^\pi\in\cN_{\xi,x}^\pi$,
\begin{align}
(D^\pi)\widehat{\hspace{2mm}}(\xi,x)
&=
c_+(\xi)\frac{1+s_p(x)}{2}+c_-(\xi)\frac{1-s_p(x)}{2}
+(Y^\pi)\widehat{\hspace{2mm}}(\xi,x),
\label{eq:mis-D-4}
\\
(Y^\pi)\widehat{\hspace{2mm}}(\xi,x)
&=
\lim_{n\to\infty}\sum_{k=1}^{m_n}
c_{n,k}(\xi)(H_{n,k}^\pi)\widehat{\hspace{2mm}}(x).
\label{eq:mis-D-5}
\end{align}
Hence every coset $D^\pi\in\widetilde{\cN}_{\xi,x}^\pi\setminus\cN_{\xi,x}^\pi$
is of the form $D^\pi=(\delta I)^\pi+O_{\xi,x}^\pi$ with $\delta\in\mC\setminus\{0\}$
and $O_{\xi,x}^\pi\in\cN_{\xi,x}^\pi$. But then $D^\pi-O_{\xi,x}^\pi=(\delta I)^\pi$ is
invertible in $\cD^\pi$ and this contradicts the maximality of
$\widetilde{\cN}_{\xi,x}^\pi$. Thus $\cN_{\xi,x}^\pi$ is a maximal ideal
of $\cD^\pi$ for $(\xi,x)\in\mathfrak{M}$, and therefore $M(\cD^\pi)=\mathfrak{M}$.
\end{proof}
Theorem~\ref{th:maximal-ideal-space-Dpi} immediately implies the following.
\begin{cor}\label{co:mis-D}
A coset $D^\pi\in\cD^\pi$ given by \eqref{eq:mis-D-2}--\eqref{eq:mis-D-3}
is invertible in the Banach algebra $\cD^\pi$ if and only if the Gelfand
transform $(D^\pi)\widehat{\hspace{2mm}}(\xi,x)$ given by
\eqref{eq:mis-D-4}--\eqref{eq:mis-D-5} does not vanish
for every $(\xi,x)\in\mathfrak{M}$.
\end{cor}
\subsection{Maximal ideal space of $\cZ^\pi$}
Now we are ready to describe the maximal ideal space of $\cZ^\pi$.
\begin{thm}\label{th:maximal-ideal-space-Zpi}
The maximal ideal space $M(\cZ^\pi)$ of the commutative Banach algebra $\cZ^\pi$
is homeomorphic to the set $\{-\infty,+\infty\}\cup(\Delta\times\mR)$.
\end{thm}
\begin{proof}
From \eqref{eq:S-R-relation} it easily follows that any coset $Z^\pi\in\cZ^\pi$
is of the form \eqref{eq:mis-D-2} where now $c_\pm$ are complex constants. Hence,
Corollaries~\ref{co:algebra-Api}(b) and \ref{co:mis-D} imply that $Z^\pi$
is invertible in $\cD^\pi$ if and only if
\begin{equation}\label{eq:mis-Z-1}
(Z^\pi)\widehat{\hspace{2mm}}(\xi,x)=
c_+(1+s_p(x))/2+c_-(1-s_p(x))/2
+(Y^\pi)\widehat{\hspace{2mm}}(\xi,x)\ne 0
\end{equation}
for every $(\xi,x)\in\mathfrak{M}$. If \eqref{eq:mis-Z-1} holds, then
$(Z^\pi)^{-1}=(d_+P_+)^\pi+(d_-P_-)^\pi+G^\pi$,
where $d_\pm\in SO(\mR_+)$ and $G^\pi$ is a coset of the form
\eqref{eq:mis-D-3}.

From Corollary~\ref{co:algebra-Api}(b)
we know that $(H_{n,k}^\pi)\widehat{\hspace{2mm}}(\pm\infty)=0$ for every
ope\-rator $H_{n,k}$ in the representation \eqref{eq:mis-D-3}. Hence for every
coset $Y^\pi$ of the form \eqref{eq:mis-D-3}, we deduce from \eqref{eq:mis-D-5} that
\begin{equation}\label{eq:mis-Z-2}
(Y^\pi)\widehat{\hspace{2mm}}(\xi,\pm\infty)=0\quad\mbox{for}\quad \xi\in M(SO(\mR_+)).
\end{equation}

Taking the Gelfand transform of the coset
\[
I^\pi=Z^\pi(Z^\pi)^{-1}=(c_+d_+P_+)^\pi+(c_-d_-P_-)^\pi+Q^\pi,
\]
where $Q^\pi$ is of the form \eqref{eq:mis-D-3}, at the points $(\xi,\pm\infty)$ for
$\xi\in M(SO(\mR_+))$, from Corollary~\ref{co:algebra-Api}(b)
and \eqref{eq:mis-Z-2} we see that $c_\pm d_\pm(\xi)=1$ for
all $\xi\in M(SO(\mR_+))$. Therefore $d_\pm=c_\pm^{-1}
\in\mC\setminus\{0\}$, whence
$(Z^\pi)^{-1}=c_+^{-1}P_+^\pi+c_-^{-1}P_-^\pi+G^\pi\in\cZ^\pi$.
Thus $Z^\pi\in\cZ^\pi$ is invertible in $\cZ^\pi$ if and only if
\eqref{eq:mis-Z-1} holds for all $(\xi,x)\in\mathfrak{M}$.

From \eqref{eq:mis-D-4} and \eqref{eq:mis-Z-2} it follows that
only one element in $M(\cZ^\pi)$ corresponds to any pair in the set
$M(SO(\mR_+))\times\{\pm\infty\}$. We will denote it by $\cI_{\pm\infty}^\pi$.
Let $\cN_{\xi,x}^\pi$ be the maximal ideal of $\cD^\pi$ corresponding
to the point $(\xi,x)\in\mathfrak{M}$.
It is clear that for every $\xi\in M(SO(\mR_+))$ we have
\[
\cI_{\pm\infty}^\pi
=\cN^\pi_{\xi,\pm\infty}\cap\cZ^\pi=\big\{Z^\pi\in\cZ^\pi:(Z^\pi)\widehat{\hspace{2mm}}(\xi,\pm\infty)=0\big\}.
\]
Finally,
\[
\cI_{\xi,x}^\pi=\cN^\pi_{\xi,x}\cap\cZ^\pi=
\big\{Z^\pi\in\cZ^\pi:(Z^\pi)\widehat{\hspace{2mm}}(\xi,x)=0\big\}
\]
is a maximal ideal of the algebra
$\cZ^\pi$ for every $(\xi,x)\in\Delta\times\mR$. Thus,
$M(\cZ^\pi)=\{-\infty,+\infty\}\cup(\Delta\times\mR)$.
\end{proof}
\subsection{Fredholmness of operators of local type}
According to the proof of Theorem~\ref{th:maximal-ideal-space-Zpi},
\begin{align*}
\cI_{\pm\infty}^\pi&:=\operatorname{id}_{\cZ^\pi}
\big\{P_\mp^\pi,(gR)^\pi : \ g\in SO(\mR_+)\big\},
\\
\cI_{\xi,x}^\pi&=\big\{Z^\pi\in\cZ^\pi:
(Z^\pi)\widehat{\hspace{2mm}}(\xi,x)=0\big\}\quad\text{for }\;(\xi,x)\in\Delta\times\mR.
\end{align*}
Further,
let $\cJ_{\pm\infty}^\pi$ and $\cJ_{\xi,x}^\pi$ be the
closed two-sided ideals of the Banach algebra $\Lambda^\pi$
generated by the ideals $\cI^\pi_{\pm\infty}$ and
$\cI^\pi_{\xi,x}$ of the algebra $\cZ^\pi$, respectively,
and put
\[
\Lambda^\pi_{\pm\infty}:=\Lambda^\pi/\cJ^\pi_{\pm\infty},
\quad
\Lambda^\pi_{\xi,x}:=\Lambda^\pi/\cJ^\pi_{\xi,x}
\]
for the corresponding quotient algebras.
\begin{thm}\label{th:localization-realization}
An operator $A\in\Lambda$ is Fredholm on the space $L^p(\mR_+)$ if and only
if the following two conditions are fulfilled:
\begin{itemize}
\item[(i)]
the cosets $A^\pi+\cJ_{\pm\infty}^\pi$ are invertible
in the quotient algebras ${\Lambda}_{\pm\infty}^\pi$,
respectively;

\item[(ii)]
for every $(\xi,x)\in\Delta\times\mR$, the coset
$A^\pi+\cJ_{\xi,x}^\pi$ is invertible in the quotient
algebra ${\Lambda}_{\xi,x}^\pi$.
\end{itemize}
\end{thm}
\begin{proof}
By Lemma~\ref{le:alg-Lambda}(c), the Fredholmness of an operator $A\in\Lambda$
is equivalent to the invertibility of the coset $A^\pi=A+\cK$ in the
quotient algebra $\Lambda^\pi$. By Theorem~\ref{th:maximal-ideal-space-Zpi}, $\cZ^\pi$ is its
central subalgebra, whose maximal ideal space is homeomorphic to the set
$\{-\infty,+\infty\}\cup(\Delta\times\mR)$. Therefore, by the Allan-Douglas
local principle (Theorem~\ref{th:AllanDouglas}), the invertibility of the coset
$A^\pi$ in the quotient algebra $\Lambda^\pi$ is equivalent to the invertibility
of the local representatives $A^\pi+\cJ_{-\infty}^\pi$, $A^\pi+\cJ_{+\infty}^\pi$,
and $A^\pi+\cJ_{\xi,x}^\pi$ in the local algebras $\Lambda_{-\infty}^\pi$,
$\Lambda_{+\infty}^\pi$, and $\Lambda_{\xi,x}^\pi$ for all $(\xi,x)\in\Delta\times\mR$,
respectively.
\end{proof}
\section{Some functions in $C_b(\mR_+,V(\mR))$ and $\cE(\mR_+,V(\mR))$}
\label{sec:functions-in-fC}
\subsection{Functions $s_p$, $r_p$, and slowly oscillating functions}
In this section we will prove that certain functions, playing a
major role in the proof of the sufficiency portion of
Theorem~\ref{th:main}, belong to $\cE(\mR_+,V(\mR))$ or to $C_b(\mR_+,V(\mR))$.
\begin{lem}\label{le:g-sp-rp}
Let $g\in SO(\mR_+)$. Then the functions
\[
\fg(t,x):=g(t),
\quad
\fs_p(t,x):=s_p(x),
\quad
\fr_p(t,x):=r_p(x),
\quad (t,x)\in\mR_+\times\mR,
\]
belong to the algebra $\cE(\mR_+,V(\mR))$.
\end{lem}
\begin{proof}
The statement is obvious for $\fg$. Let us prove it for $\fs_p$ and $\fr_p$.
Since $\fs_p$ and $\fr_p$ are constant in the first variable,
the only nontrivial property is \eqref{eq:SO-V}. Because $s_p,r_p\in V(\mR)$,
we conclude that $s_p'$ and $r_p'$ belong to $L^1(\mR)$. Taking into account that for
$X=C_b(\mR), L^1(\mR)$ and $f\in X$,
\[
\|f(\cdot+h)-f(\cdot)\|_X\to 0\quad\mbox{as}\quad |h|\to 0
\]
(see, e.g., \cite[Chap.~2, Section~6]{DL93} and also \cite[Chap.~III, Section~2]{S70}
for $L^1(\mR)$), we conclude that
\[
\begin{split}
\sup_{t\in\mR_+}\|\fs_p(t,\cdot)-\fs_p^h(t,\cdot)\|_V
=&
\|s_p(\cdot+h)-s_p(\cdot)\|_{L^\infty(\mR)}
\\
&+
\|s_p'(\cdot+h)-s_p'(\cdot)\|_{L^1(\mR)}\to 0
\end{split}
\]
as $|h|\to 0$. The same property is true with $\fs_p,s_p$, and $s_p'$ replaced by $\fr_p,r_p$, and
$r_p'$, respectively. Thus \eqref{eq:SO-V} holds for $\fs_p$ and $\fr_p$.
\end{proof}
\subsection{A function in the algebra $C_b(\mR_+,V(\mR))$}
We start with the following obvious auxiliary statement.
\begin{lem}\label{le:constants}
For every $j\in\mN\cup\{0\}$, we have
\begin{align}
& C_j^\infty:=\sup_{x\in\mR}|x+i/p|^j|r_p(x)|<\infty,
\label{eq:constants-1}
\\
& C_j^1:=\int_{\mR}|x+i/p|^j|r_p(x)|<\infty,
\label{eq:constants-2}
\\
& M_0:=\sup_{x\in\mR}\big|\pi s_p(x)\big|<\infty.
\label{eq:constants-3}
\end{align}
\end{lem}
\begin{lem}\label{le:fc}
Suppose $\alpha$ is a slowly oscillating shift and $\omega$ is its exponent
function. The function
\begin{equation}\label{eq:def-c}
\fc(t,x):=e^{i\omega(t)(x+i/p)}r_p(x),
\quad
(t,x)\in\mR_+\times\mR,
\end{equation}
belongs to the algebra $C_b(\mR_+,V(\mR))$.
\end{lem}
\begin{proof}
Through the proof we will assume that $(t,x),(\tau,x)\in\mR_+\times\mR$.
Since $\omega\in SO(\mR_+)$ is real-valued, we have
\begin{equation}\label{eq:fc-1}
M_1:=\sup_{t\in\mR_+}|\omega(t)|,
\quad
M_2:=\sup_{t\in\mR_+}(-\omega(t))<\infty.
\end{equation}
Then
\begin{equation}\label{eq:fc-2}
|e^{i\omega(t)(x+i/p)}|=e^{-\omega(t)/p}\le e^{M_2/p},
\end{equation}
whence due to \eqref{eq:def-c},
\begin{equation}\label{eq:a2}
|\fc(t,x)|\le e^{M_2/p}|r_p(x)|.
\end{equation}
From this estimate and \eqref{eq:constants-1} it follows that
\begin{equation}\label{eq:fc-3}
\|\fc(t,\cdot)\|_{L^\infty(\mR)}\le e^{M_2/p}C_0^\infty.
\end{equation}
It is easy to see that
\begin{equation}\label{eq:fc-4}
\frac{\partial \fc}{\partial x}(t,x)=\big(i\omega(t)-\pi s_p(x)\big)\fc(t,x).
\end{equation}
Hence from \eqref{eq:a2}, \eqref{eq:constants-2}--\eqref{eq:constants-3}
and \eqref{eq:fc-1} we obtain
\begin{equation}\label{eq:fc-5}
V(\fc(t,\cdot))=\int_\mR\left|\frac{\partial\fc}{\partial x}(t,x)\right|dx
\le
(M_1+M_0)e^{M_2/p}C_0^1.
\end{equation}
Further, for $t,\tau\in\mR_+$, we get
\[
\fc(t,x)-\fc(\tau,x)=
i(x+i/p)\left(\int_{\omega(\tau)}^{\omega(t)}e^{i\theta(x+i/p)}d\theta\right)r_p(x).
\]
For every $\theta$ in the segment with the endpoints $\omega(\tau)$ and
$\omega(t)$, we have
\[
|e^{i\theta(x+i/p)}|=e^{-\theta/p}\le e^{M_2/p}.
\]
Hence
\begin{equation}\label{eq:fc-6}
|\fc(t,x)-\fc(\tau,x)|\le e^{M_2/p}|\omega(t)-\omega(\tau)|\,|x+i/p|\,|r_p(x)|.
\end{equation}
From \eqref{eq:constants-1} and \eqref{eq:fc-6} we get for $t,\tau\in\mR_+$,
\begin{equation}\label{eq:fc-7}
\|\fc(t,\cdot)-\fc(\tau,\cdot)\|_{L^\infty(\mR)}
\le
C_1^\infty e^{M_2/p}|\omega(t)-\omega(\tau)|.
\end{equation}
From \eqref{eq:fc-4} it follows that
\begin{align}
\frac{\partial\fc}{\partial x}(t,x)-\frac{\partial\fc}{\partial x}(\tau,x)
&=
i\big(\omega(t)-\omega(\tau)\big)\fc(t,x)
\nonumber
\\
&\quad+
\big(i\omega(\tau)-\pi s_p(x)\big)\big(\fc(t,x)-\fc(\tau,x)\big).
\label{eq:fc-8}
\end{align}
Combining \eqref{eq:constants-3} and \eqref{eq:fc-1} with inequalities
\eqref{eq:a2} and \eqref{eq:fc-6}, we deduce from \eqref{eq:fc-8} that
\[
\begin{split}
\left|\frac{\partial\fc}{\partial x}(t,x)-\frac{\partial\fc}{\partial x}(\tau,x)\right|
\le&
e^{M_2/p}|\omega(t)-\omega(\tau)|
\\
&
\times
\big(|r_p(x)|+(M_1+M_0)|x+i/p|\,|r_p(x)|\big).
\end{split}
\]
Therefore, taking into account \eqref{eq:constants-2} we infer from the latter
inequality that for $t,\tau\in\mR_+$,
\begin{align}
V(\fc(t,\cdot)-\fc(\tau,\cdot))
&=
\int_\mR\left|\frac{\partial\fc}{\partial x}(t,x)-\frac{\partial\fc}{\partial x}(\tau,x)\right|dx
\nonumber\\
&\le
(C_0^1+(M_1+M_0)C_1^1)e^{M_2/p}|\omega(t)-\omega(\tau)|.
\label{eq:fc-9}
\end{align}
Combining \eqref{eq:fc-7} and \eqref{eq:fc-9}, we arrive at the estimate
\[
\|\fc(t,\cdot)-\fc(\tau,\cdot)\|_V\le L|\omega(t)-\omega(\tau)|
\quad(t,\tau\in\mR_+),
\]
where $L:=(C_1^\infty+C_0^1+(M_1+M_0)C_1^1)e^{M_2/p}$. This inequality implies that
$\fc$ is a continuous $V(\mR)$-valued function. Moreover, it is bounded
in view of \eqref{eq:fc-3} and \eqref{eq:fc-5}. Thus $\fc\in C_b(\mR_+,V(\mR))$.
\end{proof}
\subsection{Key lemma}
The main result of this section is the following.
\begin{lem}\label{le:bxi}
Suppose $\alpha$ is a slowly oscillating shift and $\omega$ is its
exponent function. Let $\xi\in\Delta$ and
\[
\fa(t,x):=e^{i\omega(t)(x+i/p)}\big(r_p(x)\big)^2, \quad
(t,x)\in(\mR_+\times\mR)\cup(\Delta\times\mR).
\]
Then there exists a function $\fb_\xi\in\cE(\mR_+,V(\mR))$ such that
\begin{equation}\label{eq:a3}
\fa(t,x)-\fa(\xi,x)=\big(\omega(t)-\omega(\xi)\big)r_p(x)\fb_\xi(t,x),
\quad (t,x)\in\mR_+\times\mR.
\end{equation}
\end{lem}
\begin{proof}
We divide the proof into ten steps. Through the proof we will assume that
$(t,x),(\tau,x)\in\mR_+\times\mR$ and all estimates are uniform in $t,\tau,x$.

\medskip
\noindent
\textbf{1. Definition of $\fb_\xi$.}
Consider
\[
\begin{split}
&
e^{i\omega(t)(x+i/p)}-e^{i\omega(\xi)(x+i/p)}
=
i(x+i/p)\int_{\omega(\xi)}^{\omega(t)}e^{i\theta(x+i/p)}d\theta
\\
&
\quad=
\big(\omega(t)-\omega(\xi)\big)i(x+i/p)
\int_0^1\cE_\xi(x,\omega(t),y)dy,
\end{split}
\]
where
\[
\cE_\xi(x,\theta,y):=e^{i[\omega(\xi)+y(\theta-\omega(\xi))](x+i/p)},
\quad
(x,\theta,y)\in\mR\times\mR\times[0,1].
\]
Then we obtain \eqref{eq:a3} with
\begin{equation}\label{eq:bxi-1.2}
\fb_\xi(t,x):=r_p(x)\fe_\xi(t,x),
\end{equation}
\begin{equation}\label{eq:bxi-1.3}
\fe_\xi(t,x):= i(x+i/p)\int_0^1\cE_\xi(x,\omega(t),y)dy.
\end{equation}

\noindent
\textbf{2. Uniform estimate for $\fb_\xi(t,x)$.}
From Lemma~\ref{le:SO-fundamental-property} we obtain $-\omega(\xi)\le M_2$,
where $M_2$ is defined by \eqref{eq:fc-1}.
Therefore, for $\theta=\omega(\xi)+y(\omega(t)-\omega(\xi))$ and $y\in[0,1]$,
\[
|\cE_\xi(x,\omega(t),y)|=|e^{i\theta(x+i/p)}|=e^{-\theta/p}\le e^{M_2/p}.
\]
This estimate and \eqref{eq:bxi-1.2}--\eqref{eq:bxi-1.3} imply that
\begin{equation}\label{eq:bxi-2.1}
|\fb_\xi(t,x)|\le e^{M_2/p}|x+i/p|\,|r_p(x)|.
\end{equation}

\noindent
\textbf{3. Uniform estimate for $\frac{\partial\fb_\xi}{\partial x}(t,x)$.}
Differentiating \eqref{eq:bxi-1.3} we get
\[
\begin{split}
\frac{\partial\fe_\xi}{\partial x}(t,x)
=&
i\int_0^1\cE_\xi(x,\omega(t),y)dy
\\
&
+i(x+i/p)\int_0^1 i\big[\omega(\xi)+y\big(\omega(t)-\omega(\xi)\big)\big]
\cE_\xi(x,\omega(t),y)dy.
\end{split}
\]
Integrating by parts the first integral and splitting the second
integral into two integrals, we get
\begin{equation}
\frac{\partial\fe_\xi}{\partial x}(t,x)=ie^{i\omega(t)(x+i/p)}+i\omega(\xi)\fe_\xi(t,x).
\label{eq:bxi-3.1}
\end{equation}
Taking into account \eqref{eq:bxi-1.2}, the definition of $r_p$,
\eqref{eq:bxi-3.1} and \eqref{eq:def-c}, we obtain
\begin{align}
\frac{\partial\fb_\xi}{\partial x}(t,x)
&=
\frac{dr_p}{dx}(x)\fe_\xi(t,x)+r_p(x)\frac{\partial\fe_\xi}{\partial x}(t,x)
\nonumber
\\
&=\big(-\pi s_p(x)+i\omega(\xi)\big)\fb_\xi(t,x)+i\fc(t,x).
\label{eq:bxi-3.2}
\end{align}
From \eqref{eq:bxi-3.2}, \eqref{eq:constants-3}, \eqref{eq:fc-1},
\eqref{eq:a2} and \eqref{eq:bxi-2.1} it follows that
\begin{equation}\label{eq:bxi-3.3}
\left|\frac{\partial\fb_\xi}{\partial x}(t,x)\right|
\le
M_3|x+i/p|\,|r_p(x)|+e^{M_2/p}|r_p(x)|,
\end{equation}
where $M_3:=(M_0+M_1)e^{M_2/p}$.

\medskip
\noindent
\textbf{4. Uniform estimate for $\frac{\partial^2\fb_\xi}{\partial x^2}(t,x)$.}
From \eqref{eq:bxi-3.2} and \eqref{eq:fc-4} it follows that
\begin{align}
\frac{\partial^2\fb_\xi}{\partial x^2}(t,x)
=&
\big(\big(\pi r_p(x)\big)^2+\big(-\pi s_p(x)+i\omega(\xi)\big)^2\big)\fb_\xi(t,x)
\nonumber
\\
&+
\big(-2\pi s_p(x)+i\omega(\xi)+i\omega(t)\big)
i\fc(t,x).
\label{eq:bxi-4.1}
\end{align}
Taking into account \eqref{eq:constants-1}, \eqref{eq:constants-3} and
\eqref{eq:fc-1}, we infer from \eqref{eq:bxi-4.1},
\eqref{eq:a2}, and \eqref{eq:bxi-2.1} that
\begin{equation}\label{eq:bxi-4.2}
\left|\frac{\partial^2\fb_\xi}{\partial x^2}(t,x)\right|
\le
M_4|x+i/p|\,|r_p(x)|+2M_3|r_p(x)|,
\end{equation}
where $M_4:=((\pi C_0^\infty)^2+(M_0+M_1)^2)e^{M_2/p}$.

\medskip
\noindent
\textbf{5. Uniform estimate for $\fb_\xi(t,x)-\fb_\xi(\tau,x)$.}
For $y\in[0,1]$ we have
\[
\cE_\xi(x,\omega(t),y)-\cE_\xi(x,\omega(\tau),y)
=
i(x+i/p)y\int_{\omega(\tau)}^{\omega(t)}\cE_\xi(x,\theta,y)d\theta.
\]
Then, taking into account \eqref{eq:bxi-1.3}, we get
\[
\fe_\xi(t,x) -\fe_\xi(\tau,x)=-(x+i/p)^2
\int_0^1\left(y\int_{\omega(\tau)}^{\omega(t)}\cE_\xi(x,\theta,y)d\theta\right)dy.
\]
As in Step~2, we obtain
\[
|\cE_\xi(x,\theta,y)|=
|e^{i\psi(x+i/p)}|=e^{-\psi/p}\le e^{M_2/p}
\]
for $\psi=\omega(\xi)+y(\theta-\omega(\xi))$ with $y\in[0,1]$ and $\theta$
in the segment with the endpoints $\omega(\tau)$ and $\omega(t)$. Therefore
\[
\begin{split}
|\fe_\xi(t,x)-\fe_\xi(\tau,x)|
&\le
|x+i/p|^2\int_0^1 ye^{M_2/p}|\omega(t)-\omega(\tau)|dy
\\
&=(e^{M_2/p}/2)|\omega(t)-\omega(\tau)|\,|x+i/p|^2.
\end{split}
\]
Thus, by \eqref{eq:bxi-1.2},
\begin{equation}\label{eq:bxi-5.1}
|\fb_\xi(t,x)-\fb_\xi(\tau,x)|
\le
e^{M_2/p}|\omega(t)-\omega(\tau)|\,|x+i/p|^2|r_p(x)|.
\end{equation}

\noindent
\textbf{6. Uniform estimate for
$\frac{\partial\fb_\xi}{\partial x}(t,x)-\frac{\partial\fb_\xi}{\partial x}(\tau,x)$.}
From \eqref{eq:bxi-3.2} it follows that
\begin{align}
\frac{\partial\fb_\xi}{\partial x}(t,x)-\frac{\partial\fb_\xi}{\partial x}(\tau,x)
=&
\big(-\pi s_p(x)+i\omega(\xi)\big)\big(\fb_\xi(t,x)-\fb_\xi(\tau,x)\big)
\nonumber
\\
&+i\big(\fc(t,x)-\fc(\tau,x)\big).
\label{eq:bxi-6.1}
\end{align}
Taking into account \eqref{eq:constants-3}, \eqref{eq:fc-1}, from
\eqref{eq:bxi-6.1}, \eqref{eq:bxi-5.1} and \eqref{eq:fc-6} we obtain
\begin{align}
\left|\frac{\partial\fb_\xi}{\partial x}(t,x)-\frac{\partial\fb_\xi}{\partial x}(\tau,x)\right|
\le&
\big(M_3|x+i/p|^2|r_p(x)|+e^{M_2/p}|x+i/p|\,|r_p(x)|\big)
\nonumber
\\
&\times
|\omega(t)-\omega(\tau)|.
\label{eq:bxi-6.2}
\end{align}

\noindent
\textbf{7. Uniform estimate for $\fb_\xi(t,x+h)-\fb_\xi(t,x)$.}
From \eqref{eq:bxi-3.3} and \eqref{eq:constants-1} we
obtain for $h\in\mR$,
\begin{align}
|\fb_\xi &(t,x+h)-\fb_\xi(t,x)|
=
\left|\int_x^{x+h}\frac{\partial\fb_\xi}{\partial y}(t,y)dy\right|
\nonumber
\\
&\le
M_3\left|\int_x^{x+h}|y+i/p|\,|r_p(y)|dy\right|+e^{M_2/p}\left|\int_x^{x+h}|r_p(y)|dy\right|
\nonumber
\\
&\le M_5|h|,
\label{eq:bxi-7.1}
\end{align}
where $M_5:=M_3C_1^\infty+e^{M_2/p}C_0^\infty$.

\medskip
\noindent
\textbf{8. Proof of $\fb_\xi\in C_b(\mR_+,V(\mR))$.}
From \eqref{eq:bxi-2.1} and \eqref{eq:constants-1} it follows that
\begin{equation}\label{eq:bxi-8.1}
\|\fb_\xi(t,\cdot)\|_{L^\infty(\mR)}\le e^{M_2/p}C_1^\infty.
\end{equation}
Further, from \eqref{eq:bxi-3.3} and \eqref{eq:constants-2} we get
\begin{equation}\label{eq:bxi-8.2}
V(\fb_\xi(t,\cdot))=
\int_\mR\left|\frac{\partial\fb_\xi}{\partial x}(t,x)\right|dx
\le
M_3C_1^1+e^{M_2/p}C_0^1.
\end{equation}
Combining \eqref{eq:bxi-8.1} and \eqref{eq:bxi-8.2}, we obtain
\begin{align}
\|\fb_\xi\|_{C_b(\mR_+,V(\mR))}
&=
\sup_{t\in\mR}\big(\|\fb_\xi(t,\cdot)\|_{L^\infty(\mR)}+V(\fb_\xi(t,\cdot))\big)
\nonumber\\
&\le e^{M_2/p}C_1^\infty+M_3C_1^1+e^{M_2/p}C_0^1<\infty.
\label{eq:bxi-8.3}
\end{align}
From \eqref{eq:constants-1} and \eqref{eq:bxi-5.1} we see that
\begin{equation}\label{eq:bxi-8.4}
\|\fb_\xi(t,\cdot)-\fb_\xi(\tau,\cdot)\|_{L^\infty(\mR)}
\le
e^{M_2/p}C_2^\infty|\omega(t)-\omega(\tau)|.
\end{equation}
Further, from \eqref{eq:bxi-6.2} and \eqref{eq:constants-2} it follows that
for all $t,\tau\in\mR_+$,
\begin{align}
V(\fb_\xi(t,\cdot)-\fb_\xi(\tau,\cdot))
&=
\int_\mR\left|
\frac{\partial\fb_\xi}{\partial x}(t,x)-\frac{\partial\fb_\xi}{\partial x}(\tau,x)
\right|dx
\nonumber
\\
&\le
(M_3C_2^1+e^{M_2/p}C_1^1)|\omega(t)-\omega(\tau)|.
\label{eq:bxi-8.5}
\end{align}
Combining \eqref{eq:bxi-8.4} and \eqref{eq:bxi-8.5} we see that for all
$t,\tau\in\mR_+$,
\[
\|\fb_\xi(t,\cdot)-\fb_\xi(\tau,\cdot)\|_V\le M_6|\omega(t)-\omega(\tau)|,
\]
where $M_6:=e^{M_2/p}C_2^\infty+M_3C_2^1+e^{M_2/p}C_1^1$. From this inequality
it follows that $\fb_\xi$ is a continuous $V(\mR)$-valued function. Moreover,
it is bounded in view of \eqref{eq:bxi-8.3}. Thus, $\fb_\xi\in C_b(\mR_+,V(\mR))$.

\medskip
\noindent
\textbf{9. Proof of $\fb_\xi\in SO(\mR_+,V(\mR))$.}
Estimate \eqref{eq:bxi-8.4} immediately implies that
\[
\operatorname{cm}_r^C(\fb_\xi)
\le
e^{M_2/p}C_2^\infty
\operatorname{osc}(\omega,[r,2r]),
\quad r\in\mR_+.
\]
Since $\omega\in SO(\mR_+)$, from this estimate we obtain
\[
\lim_{r\to s}\operatorname{cm}_r^C(\fb_\xi)
=
\lim_{r\to s}\operatorname{osc}(\omega,[r,2r])=0
\quad
(s\in\{0,\infty\}).
\]
Thus, taking into account the result of Step~8, we conclude that
$\fb_\xi$ belongs to the algebra $SO(\mR_+,V(\mR))$.

\medskip
\noindent
\textbf{10. Proof of $\fb_\xi\in\cE(\mR_+,V(\mR))$.}
From \eqref{eq:bxi-7.1} it follows that for $h\in\mR$,
\begin{equation}\label{eq:bxi-10.1}
\sup_{t\in\mR_+}\|\fb_\xi(t,\cdot)-\fb_\xi^h(t,\cdot)\|_{L^\infty(\mR)}
\le M_5|h|.
\end{equation}
On the other hand, from \eqref{eq:bxi-4.2} we obtain for $t\in\mR_+$ and
$h>0$,
\begin{align}
V(\fb_\xi & (t,\cdot)-\fb_\xi^h(t,\cdot))
=\int_\mR\left|
\frac{\partial\fb_\xi}{\partial x}(t,x+h)-\frac{\partial\fb_\xi}{\partial x}(t,x)
\right|dx
\nonumber
\\
&=
\int_\mR\left|\int_x^{x+h}\frac{\partial^2\fb_\xi}{\partial y^2}(t,y)dy\right|dx
\le
\int_\mR\int_x^{x+h}\left|\frac{\partial^2\fb_\xi}{\partial y^2}(t,y)\right|dy\,dx
\nonumber
\\
&\le M_4\int_\mR\int_x^{x+h}|y+i/p|\,|r_p(y)|dy\,dx
+2M_3\int_\mR\int_x^{x+h}|r_p(y)|dy\,dx.
\label{eq:bxi-10.2}
\end{align}
Changing the order of integration and taking into account
\eqref{eq:constants-2}, we get for $h\in\mR$ and $j\in\{0,1\}$,
\begin{align}
\int_\mR  \int_x^{x+h}  |y+i/p|^j|r_p(y)|dy\,dx
&=
\int_\mR\int_{y-h}^y|y+i/p|^j|r_p(y)|dx\,dy
\nonumber\\
&=
h\int_\mR|y+i/p|^j|r_p(y)|dy
=
C_j^1 h.
\label{eq:bxi-aux}
\end{align}
Combining \eqref{eq:bxi-10.2} and \eqref{eq:bxi-aux}, we see that
\begin{equation}\label{eq:bxi-10.3}
V(\fb_\xi(t,\cdot)-\fb_\xi^h(t,\cdot))\le M_7h \quad (h>0),
\end{equation}
where $M_7:=M_4C_1^1+2M_3C_0^1$. Analogously it can be shown that
\begin{equation}\label{eq:bxi-10.4}
V(\fb_\xi(t,\cdot)-\fb_\xi^h(t,\cdot))\le M_7(-h)\quad (h<0).
\end{equation}
From \eqref{eq:bxi-10.3} and \eqref{eq:bxi-10.4} we get for $h\in\mR$,
\begin{equation}\label{eq:bxi-10.5}
\sup_{t\in\mR_+}V\big(\fb_\xi(t,\cdot)-\fb_\xi^h(t,\cdot)\big)\le M_7|h|.
\end{equation}
Combining \eqref{eq:bxi-10.1} with \eqref{eq:bxi-10.5}, we arrive at the
equality
\begin{equation}\label{eq:bxi-10.6}
\lim_{|h|\to 0}\sup_{t\in\mR_+}\|\fb_\xi(t,\cdot)-\fb_\xi^h(t,\cdot)\|_V=0.
\end{equation}
From Step~9 and equality \eqref{eq:bxi-10.6} it finally follows that
$\fb_\xi\in\cE(\mR_+,V(\mR))$.
\end{proof}
\section{Sufficient conditions for Fredholmness}\label{sec:sufficiency}
\subsection{Invertibility in the quotient algebras $\Lambda_{+\infty}^\pi$
and $\Lambda_{-\infty}^\pi$}
\begin{thm}\label{th:suf-i}
Suppose $a,b,c,d\in SO(\mR_+)$, $\alpha\in SOS(\mR_+)$, and the operator $N$
is given by \eqref{eq:def-N}.
\begin{enumerate}
\item[{\rm(a)}]
If the operator $A_+:=aI-bW_\alpha$ is invertible on the space
$L^p(\mR_+)$, then the coset $N^\pi+\cJ_{+\infty}^\pi$ is
invertible in the quotient algebra $\Lambda_{+\infty}^\pi$.

\item[{\rm(b)}]
If the operator $A_-:=cI-dW_\alpha$ is invertible on the space
$L^p(\mR_+)$, then the coset $N^\pi+\cJ_{-\infty}^\pi$ is
invertible in the quotient algebra $\Lambda_{-\infty}^\pi$.
\end{enumerate}
\end{thm}
\begin{proof}
(a) If $A_+$ is invertible in $\cB$, then from Theorem~\ref{th:FO}
it follows that $A_+^{-1}\in\cF$. Then from Theorem~\ref{th:embeddings}
we see that $A_\pm\in\Lambda$ and $A_+$ is invertible in the algebra
$\Lambda$. Hence the coset $A_+^\pi=A_++\cK$ is invertible in the
quotient algebra $\Lambda^\pi$, which implies the invertibility of
the coset $A_+^\pi+\cJ_{+\infty}^\pi$ in the quotient algebra
$\Lambda_{+\infty}^\pi$. It remains to observe that
\[
N^\pi+\cJ_{+\infty}^\pi=(A_+P_++A_-P_-)^\pi+\cJ_{+\infty}^\pi
=A_+^\pi+\cJ_{+\infty}^\pi.
\]
Part (a) is proved. The proof of part (b) is analogous.
\end{proof}
\subsection{Invertibility in the quotient algebras $\Lambda_{\xi,x}^\pi$ with $(\xi,x)\in\Delta\times\mR$}
\begin{lem}\label{le:Op-bxi-in-Lambda}
Suppose $\alpha$ is a slowly oscillating shift and $\omega$ is its
exponent function. For $\xi\in\Delta$, let the function $\fb_\xi$ be
defined by \eqref{eq:bxi-1.2}--\eqref{eq:bxi-1.3}.
Then the operator $\Phi^{-1}\operatorname{Op}(\fb_\xi)\Phi$
belongs to the algebra $\Lambda$.
\end{lem}
\begin{proof}
Let $g\in SO(\mR_+)$ and
\[
\fg_p(t,x):=g(t)r_p(x),
\quad
\fs_p(t,x)=s_p(x),
\quad
(t,x)\in\mR_+\times\mR.
\]
Lemmas~\ref{le:g-sp-rp} and \ref{le:bxi} imply that the functions $\fs_p,\fg_p$,
and $\fb_\xi$ belong to the algebra $\cE(\mR_+,V(\mR))$.
From this observation and Theorem~\ref{th:comp-commutators-PDO} it follows that
\begin{equation}\label{eq:a5}
\begin{split}
&
\operatorname{Op}(\fs_p)
\operatorname{Op}(\fb_\xi)
-
\operatorname{Op}(\fb_\xi)
\operatorname{Op}(\fs_p)
\in\cK(L^p(\mR_+,d\mu)),
\\
&
\operatorname{Op}(\fg_p)
\operatorname{Op}(\fb_\xi)
-
\operatorname{Op}(\fb_\xi)
\operatorname{Op}(\fg_p)
\in\cK(L^p(\mR_+,d\mu)).
\end{split}
\end{equation}
Since $\cE(\mR_+,V(\mR))\subset C_b(\mR_+,V(\mR))$, we infer from
Theorem~\ref{th:boundedness-PDO} that
$B_\xi:=\Phi^{-1}\operatorname{Op}(\fb_\xi)\Phi\in\cB$.
Then relations \eqref{eq:a5} and the equalities
\[
S=\Phi^{-1}\operatorname{Co}(s_p)\Phi=\Phi^{-1}\operatorname{Op}(\fs_p)\Phi,
\quad
gR=g\Phi^{-1}\operatorname{Co}(r_p)\Phi=\Phi^{-1}\operatorname{Op}(\fg_p)\Phi
\]
(see Theorem~\ref{th:algebra-A}(b)) imply that $SB_\xi-B_\xi S\in\cK$
and $gRB_\xi-B_\xi gR\in\cK$. Hence $B_\xi\in\Lambda$.
\end{proof}
\begin{lem}\label{le:WR2}
Suppose $\alpha$ is a slowly oscillating shift and $\omega$ is its exponent
function. If $(\xi,x)\in\Delta\times\mR$, then
\[
(W_\alpha R^2)^\pi-e^{i\omega(\xi)(x+i/p)}(r_p(x))^2 I^\pi\in\cJ_{\xi,x}^\pi.
\]
\end{lem}
\begin{proof}
From \cite[formula 3.194.4]{GR07} it follows that for $k>0$ and $y\in\mR$,
\[
\frac{1}{\pi i}\int_{\mR_+}\frac{t^{1/p}}{1+kt}t^{-iy}\frac{dt}{t}
=
\frac{1}{k^{1/p-iy}}\cdot\frac{1}{i\sin[\pi(1/p-iy)]}
=e^{i(y+i/p)\log k}r_p(y).
\]
Taking the inverse Mellin transform, we get for $k,t\in\mR_+$,
\[
\frac{1}{\pi i}\frac{t^{1/p}}{1+kt}=\frac{1}{2\pi}\int_\mR e^{i(y+i/p)\log k}r_p(y)t^{iy}dy.
\]
Assume that $f\in C_0^\infty(\mR_+)$. Since $\alpha(\tau)=e^{\omega(\tau)}\tau$,
from the above identity it follows that for $t\in\mR_+$,
\[
\begin{split}
(\Phi W_\alpha R\Phi^{-1}f)(t)
&=
\frac{1}{\pi i}\int_{\mR_+}\frac{f(\tau)(t/\tau)^{1/p}}{\tau+\alpha(t)}d\tau
=
\frac{1}{\pi i}\int_{\mR_+}\frac{f(\tau)(t/\tau)^{1/p}}{1+e^{\omega(t)}(t/\tau)}\frac{d\tau}{\tau}
\\
&=
\frac{1}{2\pi}\int_{\mR_+}\left(\int_\mR e^{i\omega(t)(y+i/p)}r_p(y)\left(\frac{t}{\tau}\right)^{iy}dy\right)f(\tau)\frac{d\tau}{\tau}
\\
&=
\frac{1}{2\pi}\int_\mR dy
\int_{\mR_+}e^{i\omega(t)(y+i/p)}r_p(y)\left(\frac{t}{\tau}\right)^{iy}f(\tau)\frac{d\tau}{\tau}.
\end{split}
\]
Hence, for $f\in C_0^\infty(\mR_+)$, we have
$\Phi W_\alpha R\Phi^{-1} f=\operatorname{Op}(\fc)f$, where
\[
\fc(t,y):=e^{i\omega(t)(y+i/p)}r_p(y),
\quad
(t,y)\in(\mR_+\times\mR)\cup(\Delta\times\mR).
\]
By Lemma~\ref{le:fc}, this function belongs to the algebra $C_b(\mR_+,V(\mR))$.
Then from Theorem~\ref{th:boundedness-PDO} it follows that $\operatorname{Op}(\fc)$
extends to a bounded operator on $L^p(\mR_+,d\mu)$ and therefore
\begin{equation}\label{eq:WR2-1}
\Phi W_\alpha R\Phi^{-1}=\operatorname{Op}(\fc).
\end{equation}
On the other hand, from Theorem~\ref{th:algebra-A}(b) and Lemma~\ref{le:g-sp-rp}
it follows that
\begin{equation}\label{eq:WR2-2}
\Phi R\Phi^{-1}=\operatorname{Co}(r_p)=\operatorname{Op}(\fr_p),
\end{equation}
where $\fr_p(t,y)=r_p(y)$ for $(t,y)\in\mR_+\times\mR$. From \eqref{eq:WR2-1}
and \eqref{eq:WR2-2} we obtain
\begin{equation}\label{eq:WR2-3}
W_\alpha R^2=\Phi^{-1}\operatorname{Op}(\fa)\Phi,
\end{equation}
where
\[
\fa(t,y):=\fc(t,y)\fr_p(t,y)=e^{i\omega(t)(y+i/p)}(r_p(y))^2,
\
(t,y)\in(\mR_+\times\mR)\cup(\Delta\times\mR).
\]
Let us represent this function in the form
\begin{align}
\fa(t,y)
&=
\fa(t,y)-\fa(\xi,y)+\fc(\xi,y)r_p(y)
\nonumber
\\
&=
\big(\fa(t,y)-\fa(\xi,y)\big)+
\big(\fc(\xi,y)-\fc(\xi,x)\big)\fr_p(t,y)
\nonumber
\\
&\quad+\fc(\xi,x)\big(\fr_p(t,y)-\fr_p(t,x)\big)+\fa(\xi,x).
\label{eq:WR2-4}
\end{align}
From Lemma~\ref{le:bxi} and Theorem~\ref{th:algebra-A}(b) it follows that
\[
\Phi^{-1}\operatorname{Op}\big(\fa-\fa(\xi,\cdot)\big)\Phi
=
\big(\omega-\omega(\xi)\big)R\Phi^{-1}\operatorname{Op}(\fb_\xi)\Phi,
\]
where $\fb_\xi\in\cE(\mR_+,V(\mR))$. By Lemma~\ref{le:Op-bxi-in-Lambda}, the operator
$\Phi^{-1}\operatorname{Op}(\fb_\xi)\Phi$ belongs to the algebra $\Lambda$.
It is easy to see that by Corollary~\ref{co:algebra-Api}(b),
\[
\big(\big[\big(\omega-\omega(\xi)\big)R\big]^\pi\big)\widehat{\hspace{2mm}}(\xi,x)=
\big(\omega(\xi)-\omega(\xi)\big)r_p(x)=0.
\]
Therefore $[(\omega-\omega(\xi))R]^\pi\in\cI_{\xi,x}^\pi$ and thus
\begin{equation}\label{eq:WR2-5}
\big[\Phi^{-1}\operatorname{Op}\big(\fa-\fa(\xi,\cdot)\big)\Phi\big]^\pi
=
\big[\big(\omega-\omega(\xi)\big)R\Phi^{-1}\operatorname{Op}(\fb_\xi)\Phi\big]^\pi\in\cJ_{\xi,x}^\pi.
\end{equation}
Since $\omega(\xi)\in\mR$, we see that $\operatorname{Re}(\pi-i\omega(\xi))=\pi\in(0,2\pi)$.
From Theorem~\ref{th:algebra-A}(b) we conclude that
\[
\Phi^{-1}\operatorname{Op}(\fc(\xi,\cdot))\Phi=
\Phi^{-1}\operatorname{Co}(r_{p,\pi-i\omega(\xi)})\Phi=
R_{\pi-i\omega(\xi)}\in\cA.
\]
Hence, by Corollary~\ref{co:algebra-Api}(b),
\[
\begin{split}
\big(\big[\Phi^{-1}\operatorname{Op}\big(\fc(\xi,\cdot)-\fc(\xi,x)\big)\Phi\big]^\pi\big)
\widehat{\hspace{2mm}}(\xi,x)
&=
\big([R_{\pi-i\omega(\xi)}-\fc(\xi,x)I]^\pi\big)\widehat{\hspace{2mm}}(\xi,x)
\\
&=
r_{p,\pi-i\omega(\xi)}(x)-\fc(\xi,x)
=0.
\end{split}
\]
Therefore
\[
\big[\Phi^{-1}\operatorname{Op}\big(\fc(\xi,\cdot)-\fc(\xi,x)\big)\Phi\big]^\pi
=
[R_{\pi-i\omega(\xi)}-\fc(\xi,x)I]^\pi\in\cI_{\xi,x}^\pi
\]
and thus
\begin{equation}\label{eq:WR2-6}
\big[\Phi^{-1}\operatorname{Op}\big[\big(\fc(\xi,\cdot)-\fc(\xi,x)\big)\fr_p\big]\Phi\big]^\pi
=
\big[\big(R_{\pi-i\omega(\xi)}-\fc(\xi,x)I\big)R\big]^\pi\in\cJ_{\xi,x}^\pi.
\end{equation}
Finally, in view of Corollary~\ref{co:algebra-Api}(b),
\[
\big([\Phi^{-1}\operatorname{Op}(\fr_p-r_p(x))\Phi]^\pi\big)\widehat{\hspace{2mm}}(\xi,x)
=
\big([R-r_p(x)I]^\pi\big)\widehat{\hspace{2mm}}(\xi,x)=r_p(x)-r_p(x)=0.
\]
Hence
\begin{equation}\label{eq:WR2-7}
\fc(\xi,x)[\Phi^{-1}\operatorname{Op}(\fr_p-r_p(x))\Phi]^\pi\in\cI_{\xi,x}^\pi\subset\cJ_{\xi,x}^\pi.
\end{equation}
Combining \eqref{eq:WR2-3}--\eqref{eq:WR2-7}, we arrive at
\[
(W_\alpha R^2)^\pi-e^{i\omega(\xi)(x+i/p)}(r_p(x))^2 I^\pi
=
\big[\Phi^{-1}\operatorname{Op}(\fa)\Phi-\fa(\xi,x)I\big]^\pi\in\cJ_{\xi,x}^\pi,
\]
which finishes the proof.
\end{proof}
\begin{thm}\label{th:suf-ii}
Suppose $a,b,c,d\in SO(\mR_+)$, $\alpha\in SOS(\mR_+)$, and the operator $N$
is given by \eqref{eq:def-N}.
If $n_\xi(x)\ne 0$ for some $(\xi,x)\in\Delta\times\mR$, where the function
$n_\xi$ is defined by \eqref{eq:def-n}, then the coset
$N^\pi+\cJ_{\xi,x}^\pi$ is invertible in the quotient algebra $\Lambda_{\xi,x}^\pi$.
\end{thm}
\begin{proof}
Fix $(\xi,x)\in\Delta\times\mR$ and consider the operators
\begin{equation}\label{eq:a4}
H_\pm:=\frac{1\pm s_p(x)}{2}\left(\frac{1}{r_p(x)}\right)^2R^2.
\end{equation}
Then from Corollary~\ref{co:algebra-Api}(b) it follows that
\[
(H_\pm^\pi)\widehat{\hspace{2mm}}(\xi,x)=\frac{1\pm s_p(x)}{2}.
\]
Therefore $(P_\pm-H_\pm)^\pi\in\cI_{\xi,x}^\pi$ and
\[
\big[
(aI-bW_\alpha)(P_+-H_+)+(cI-dW_\alpha)(P_--H_-)
\big]^\pi\in\cJ_{\xi,x}^\pi,
\]
whence
\begin{equation}\label{eq:suf-ii-1}
N^\pi+\cJ_{\xi,x}^\pi=
\big[
(aI-bW_\alpha)H_++(cI-dW_\alpha)H_-
\big]^\pi+\cJ_{\xi,x}^\pi.
\end{equation}
Since $H_\pm\in\operatorname{id}_\cA\{R\}\subset\cA$, we infer from
Corollary~\ref{co:comp-commutators-algebra} that
\begin{equation}\label{eq:suf-ii-2}
(W_\alpha H_+)^\pi=(H_+W_\alpha)^\pi,
\quad
(W_\alpha H_-)^\pi=(H_-W_\alpha)^\pi.
\end{equation}
Taking into account Corollary~\ref{co:algebra-Api}(b), it is easy to see that
\begin{equation}\label{eq:suf-ii-3}
\big[(bH_+)^\pi-(b(\xi)H_+)^\pi\big]\widehat{\hspace{2mm}}(\xi,x)=0,
\quad
\big[(dH_-)^\pi-(d(\xi)H_-)^\pi\big]\widehat{\hspace{2mm}}(\xi,x)=0
\end{equation}
and
\[
\big[(aH_+)^\pi-(a(\xi)H_+)^\pi\big]\widehat{\hspace{2mm}}(\xi,x)=0,
\quad
\big[(cH_-)^\pi-(c(\xi)H_-)^\pi\big]\widehat{\hspace{2mm}}(\xi,x)=0.
\]
Hence
\begin{equation}\label{eq:suf-ii-4}
(aH_+)^\pi-(a(\xi)H_+)^\pi,\
(cH_-)^\pi-(c(\xi)H_-)^\pi
\in\cI_{\xi,x}^\pi\subset\cJ_{\xi,x}^\pi.
\end{equation}
Taking into account \eqref{eq:suf-ii-2}--\eqref{eq:suf-ii-3}, we also see
that
\begin{align}
&
(bW_\alpha H_+)^\pi-(b(\xi)W_\alpha H_+)^\pi=[(bH_+)^\pi-(b(\xi)H_+)^\pi]W_\alpha^\pi\in\cJ_{\xi,x}^\pi,
\label{eq:suf-ii-5}
\\
&
(dW_\alpha H_-)^\pi-(d(\xi)W_\alpha H_-)^\pi=[(dH_-)^\pi-(d(\xi)H_-)^\pi]W_\alpha^\pi\in\cJ_{\xi,x}^\pi.
\label{eq:suf-ii-6}
\end{align}
From \eqref{eq:suf-ii-1} and \eqref{eq:suf-ii-4}--\eqref{eq:suf-ii-6} it
follows that
\begin{equation}\label{eq:suf-ii-7}
N^\pi+\cJ_{\xi,x}^\pi=
\big(a(\xi)H_+-b(\xi)W_\alpha H_++c(\xi)H_--d(\xi)W_\alpha H_-\big)^\pi+\cJ_{\xi,x}^\pi.
\end{equation}
It is easy to see that
\[
\big(\big[(1/r_p(x))^2R^2-I\big]^\pi\big)\widehat{\hspace{2mm}}(\xi,x)=0.
\]
Hence
\begin{equation}\label{eq:suf-ii-8}
H_\pm^\pi-\frac{1\pm s_p(x)}{2}I^\pi\in\cI_{\xi,x}^\pi\subset\cJ_{\xi,x}^\pi.
\end{equation}
By Lemma~\ref{le:WR2} and \eqref{eq:a4},
\begin{equation}\label{eq:suf-ii-9}
(W_\alpha H_\pm)^\pi-e^{i\omega(\xi)(x+i/p)}\frac{1\pm s_p(x)}{2}\,I^\pi\in\cJ_{\xi,x}^\pi.
\end{equation}
Combining \eqref{eq:suf-ii-7}--\eqref{eq:suf-ii-9}, we arrive at the relation
\[
N^\pi+\cJ_{\xi,x}^\pi=n_\xi(x)I^\pi+\cJ_{\xi,x}^\pi,
\]
where $n_\xi(x)$ is given by \eqref{eq:def-n}. If $n_\xi(x)\ne 0$, then one
can check straightforwardly that $(1/n_\xi(x))I^\pi+\cJ_{\xi,x}^\pi$ is the
inverse of the coset $N^\pi+\cJ_{\xi,x}^\pi$ in the quotient algebra
$\Lambda_{\xi,x}^\pi$.
\end{proof}
\subsection{Proof of Theorem~\ref{th:main}}
If condition (i) of Theorem~\ref{th:main}
is fulfilled, then by Theorem~\ref{th:suf-i} the cosets $N^\pi+\cJ_{\pm\infty}^\pi$
are invertible in the quotient algebras $\Lambda_{\pm\infty}^\pi$, respectively.
On the other hand, if condition (ii) of Theorem~\ref{th:main} holds,
then in view of Theorem~\ref{th:suf-ii}, the coset $N^\pi+\cJ_{\xi,x}^\pi$
is invertible in the quotient algebra $\Lambda_{\xi,x}^\pi$ for every
pair $(\xi,x)\in\Delta\times\mR$. Then, by Theorem~\ref{th:localization-realization},
the operator $N\in\Lambda$ is Fredholm.
\qed

\end{document}